\documentclass[a4paper,11pt]{article}
\usepackage{amsmath}
\usepackage{amssymb}
\usepackage{theorem}
\usepackage{pstricks}
\usepackage{euscript}
\usepackage{epic,eepic}
\usepackage{graphicx}
\PassOptionsToPackage{normalem}{ulem}
\PassOptionsToPackage{normalem}{ulem}

\newcommand{\menge}[2]{\big\{{#1} \mid {#2}\big\}}

\newcommand{\emp}{\ensuremath{{\varnothing}}}

\newcommand{\infconv}{\ensuremath{\mbox{\small$\,\square\,$}}}
\newcommand{\scal}[2]{\left\langle{#1}\mid {#2} \right\rangle} 
\newcommand{\pscal}[2]{\langle\langle{#1}\mid{#2}\rangle\rangle} 

\newcommand{\vuo}{\ensuremath{\mbox{\footnotesize$\square$}}}

\newcommand{\HH}{\ensuremath{\mathcal H}}
\newcommand{\GG}{\ensuremath{\mathcal G}}

\newcommand{\BL}{\ensuremath{\EuScript B}\,}
\newcommand{\BP}{\ensuremath{\EuScript P}}

\newcommand{\KKK}{\ensuremath{\boldsymbol{\mathcal K}}}

\newcommand{\RR}{\ensuremath{\mathbb R}}

\newcommand{\RP}{\ensuremath{\left[0,+\infty\right[}}
\newcommand{\RPP}{\ensuremath{\,\left]0,+\infty\right[}}

\newcommand{\NN}{\ensuremath{\mathbb N}}

\newcommand{\dom}{\ensuremath{\operatorname{dom}}}

\newcommand{\prox}{\ensuremath{\operatorname{prox}}}

\newcommand{\ran}{\ensuremath{\operatorname{ran}}}
\newcommand{\zer}{\ensuremath{\operatorname{zer}}}
\newcommand{\gra}{\ensuremath{\operatorname{gra}}}

\newcommand{\vv}{\ensuremath{\boldsymbol{v}}}

\newcommand{\xx}{\ensuremath{\boldsymbol{x}}}
\newcommand{\xxx}{\ensuremath{\overline{\boldsymbol{x}}}}
\newcommand{\pp}{\ensuremath{\boldsymbol{p}}}
\newcommand{\qq}{\ensuremath{\boldsymbol{q}}}
\newcommand{\yy}{\ensuremath{\boldsymbol{y}}}

\newcommand{\ee}{\ensuremath{\boldsymbol{e}}}

\newcommand{\bb}{\ensuremath{\boldsymbol{b}}}
\newcommand{\uu}{\ensuremath{\boldsymbol{u}}}
\newcommand{\cc}{\ensuremath{\boldsymbol{c}}}
\newcommand{\dd}{\ensuremath{\boldsymbol{d}}}
\newcommand{\aaa}{\ensuremath{\boldsymbol{a}}}
\newcommand{\ww}{\ensuremath{\boldsymbol{w}}}
\newcommand{\BB}{\ensuremath{\boldsymbol{B}}}

\newcommand{\UU}{\ensuremath{\boldsymbol{U}}}

\newcommand{\E}{\ensuremath{\mathbf{E}}}

\newcommand{\AAA}{\ensuremath{\boldsymbol{A}}}
\newcommand{\BBB}{\ensuremath{\boldsymbol{B}}}

\newcommand{\FF}{\ensuremath{\boldsymbol{\mathcal{F}}}}

\newcommand{\Id}{\ensuremath{\operatorname{Id}}}

\newcommand{\weakly}{\ensuremath{\rightharpoonup}}

\newcommand{\pinf}{\ensuremath{+\infty}}

\newtheorem{theorem}{Theorem}[section]
\newtheorem{lemma}[theorem]{Lemma}

\newtheorem{corollary}[theorem]{Corollary}
\newtheorem{proposition}[theorem]{Proposition}

\theoremstyle{plain}{\theorembodyfont{\rmfamily}
}
\theoremstyle{plain}{\theorembodyfont{\rmfamily}
}
\theoremstyle{plain}{\theorembodyfont{\rmfamily}
}
\theoremstyle{plain}{\theorembodyfont{\rmfamily}
\newtheorem{example}[theorem]{Example}}
\theoremstyle{plain}{\theorembodyfont{\rmfamily}
\newtheorem{problem}[theorem]{Problem}}
\theoremstyle{plain}{\theorembodyfont{\rmfamily}
\newtheorem{remark}[theorem]{Remark}}
\theoremstyle{plain}{\theorembodyfont{\rmfamily}
}

\definecolor{labelkey}{rgb}{0,0.08,0.45}
\definecolor{refkey}{rgb}{0,0.6,0.0}
\definecolor{Brown}{rgb}{0.45,0.0,0.05}
\definecolor{dgreen}{rgb}{0.00,0.49,0.00}
\definecolor{dblue}{rgb}{0,0.08,0.75}
\numberwithin{equation}{section}

\begin{document}
\title{\sffamily\huge 
Almost sure convergence of the forward-backward-forward splitting algorithm }
\author{
B$\grave{\text{\u{a}}}$ng C\^ong V\~u\\[5mm]
 LCSL, Istituto Italiano di Tecnologia\\
       and Massachusetts Institute of Technology,\\
       Bldg. 46-5155, 77 Massachusetts Avenue, Cambridge, MA 02139, USA\\
({email: Cong.Bang@iit.it})} 
\date{}
\maketitle
\begin{abstract}
In this paper, we propose a stochastic forward-backward-forward splitting algorithm and 
prove its almost sure  weak convergence in real separable Hilbert spaces. Applications to 
composite monotone inclusion and minimization problems are demonstrated.
\end{abstract}

{\bf Keywords:} 
monotone inclusion,
monotone operator,
operator splitting,
Lipschitzian operators,
forward-backward-forward algorithm,
composite operator,
duality,
primal-dual algorithm

{\bf Mathematics Subject Classifications (2010)}: 47H05, 49M29, 49M27, 90C25 

\section{Introduction}
\label{intro}
Forward-backward-forward splitting algorithm was firstly proposed in \cite{Tseng00} for solving the problem
of finding a zero point of the sum of a maximally monotone operator $A\colon\HH\to 2^{\HH}$ and a monotone Lipschitzian 
operator $C\colon\HH\to\HH$, where $\HH$ is a real Hilbert space. This splitting algorithm plays a role in solving a large class 
of composite monotone inclusions \cite{siop2} and monotone inclusions involving the parallel sums \cite{Botb, plc6, Combettes13, Bang14} as well as applications to conposite convex optimization problem involving the infimal-convolutions \cite{siop2, Botb, plc6, Bang14, Luis11}. However, these works are limitted to deterministic setting.

Very recently, we have found out in the literature that there appears the study of 
some splitting algorithms for solving monotone inclusions in the stochastic setting as in \cite{plc14, LSB14,JCP14},
and  primal-dual splitting algorithm for composite monotone inclusions in \cite{plc14,JCP14}.
Some iterations in \cite{plc14, LSB14,JCP14}
are designed for monotone inclusions involving cocoercive operators. 
For solving monotone inclusions involving Lipschitzian monotone operators,  
one can often  use the iterations which has  the structure
of the forward-backward-forward splitting methods as cited above, but  
the  convergence of their proposed methods is no longer available, in the literature, in the stochastic setting.

The objective of this note is to study the convergence of the forward-backward-forward 
splitting in the stochastic setting for monotone inclusions  involving  Lipschitzian monotone operators
as well as for composite monotone inclusions involving  parallel sums.

In Section \ref{s:reca}, we recall some notations, background  and preliminary results. We prove 
the almost sure convergence of the stochastic forward-backward-forward splitting algorithm in Section \ref{s:sfbf}.
In the last section, we provide applications to composite monotone inclusions involving the parallel sums as well as minimization 
problems involving infimal convolutions.

\section{Notation--background and premilary results}
\label{s:reca}
Throughout, 
$\HH$, $\GG$, and $(\GG_i)_{1\leq i\leq m}$ are real  separable
Hilbert spaces. Their scalar products and  associated norms are respectively denoted 
by $\scal{\cdot}{\cdot}$ and  $\|\cdot\|$. 
We denote by $\BL(\HH,\GG)$ the space of bounded linear operators 
from $\HH$ to $\GG$. The adjoint of $L\in\BL(\HH,\GG)$ is denoted by $L^*$.
We set $\BL(\HH)=\BL(\HH,\HH)$.
$\Id$ denotes the identity operator.
The symbols $\weakly$ and $\to$ denote weak and strong convergence, respectively. We denote  
by $\ell_+^1(\NN)$ the set of summable sequences in  $\RP$.
The class of all proper lower semicontinuous convex functions from $\HH$
to $\left]-\infty,+\infty \right]$
is denoted by $\Gamma_0(\HH)$.
Let $M_1$ and $M_2$ be self-adjoint operators in $\BL(\HH)$, we write 
$
M_1\succcurlyeq M_2$ if and only if $(\forall x\in\HH)\;
\scal{M_1x}{x}\geq\scal{M_2x}{x}.
$
Let $\alpha\in\left]0,+\infty\right[$. We set
\begin{equation}
\BP_{\alpha}(\HH)=\menge{M\in\BL(\HH)}{M^* =M \quad \text{and}\quad M\succcurlyeq\alpha\Id}.
\end{equation}

Let $A\colon\HH\to 2^{\HH}$ be a set-valued operator.
The domain of $A$ is $\dom A=\menge{x\in\HH}{Ax\neq\emp}$,
 and the graph of $A$ is 
$\gra A=\menge{(x,u)\in\HH\times\HH}{u\in Ax}$.
The set of zeros 
of $A$ is  $\zer A=\menge{x\in\HH}{0\in Ax}$, and the range of $A$ is
$\ran A=\menge{u\in\HH}{(\exists\; x\in\HH)\;u\in Ax}$. 
The inverse of $A$ is $A^{-1}\colon\HH\mapsto 2^{\HH}\colon u\mapsto 
\menge{x\in\HH}{u\in Ax}$, and the resolvent of $A$ is
\begin{equation}
\label{e:resolvent}
J_A=(\Id+A)^{-1}.
\end{equation}
Moreover, $A$ is monotone if 
\begin{equation}
(\forall(x,y)\in\HH\times\HH)
(\forall(u,v)\in Ax\times Ay)\quad\scal{x-y}{u-v}\geq 0,
\end{equation}
and maximally monotone if it is monotone and there exists no 
monotone operator 
$B\colon\HH\to2^\HH$ such that $\gra A\subset\gra B$ and $A\neq B$.
We say that $A$ is uniformly monotone 
at $x\in\dom A$ if there exists an 
increasing function $\phi_A\colon\left[0,+\infty\right[\to 
\left[0,+\infty\right]$ vanishing only at $0$ such that 
\begin{equation}\label{oioi}
\big(\forall u\in Ax\big)\big(\forall (y,v)\in\gra A\big)
\quad\scal{x-y}{u-v}\geq\phi_A(\|x-y\|).
\end{equation}

Given a probability space
$(\boldsymbol{\Omega, \mathcal{F},\mathsf{P}})$,
 we denote by 
$\sigma(x)$ the  $\sigma$-field generated by a random
 vector $x\colon \boldsymbol{\Omega}\to\HH$, 
where $\HH$ is endowed with the Borel $\sigma$-algebra. The 
expectation of a random variable $x$ is denoted by $\E[x]$. The conditional expectation
of $x$ given a sub-sigma algebra $\mathcal{F}\subset \boldsymbol{\mathcal{F}}$ is denoted by 
$\E[x|\mathcal{F}]$. The conditional expectation of $x$ given $y$ is denoted by $\E[x|y]$. 

\begin{lemma}{\rm\cite[Theorem 1]{Rob85}}\label{l:rob85}
Let $(\FF_n)_{n\in\NN}$ be an increasing sequence of  sub-sigma algebras of $\boldsymbol{\mathcal{F}}$.
For every $n\in\NN$, 
let $z_n$, $\xi_n$, $\zeta_n$ and $t_n$ be  
non-negative, $\FF_n$-measurable random variable such that $(\zeta_n)_{n\in\NN}$ and $(t_n)_{n\in\NN}$ are
 summable  and 
\begin{equation}\label{eq:rob}
(\forall n\in\NN)\quad \E[z_{n+1}|\FF_n] \leq (1+t_n)z_n + \zeta_n-\xi_n 
\quad \text{$\boldsymbol{\mathsf{P}}$-a.s.}
\end{equation}
Then $(z_n)_{n\in\NN}$ converges and $(\xi_n)_{n\in\NN}$ is summable $\boldsymbol{\mathsf{P}}$-a.s.
\end{lemma}

\begin{lemma}{\rm \cite[Proposition 2.3]{plc14}}
\label{p:fejer}
Let $\HH$ be a real separable Hilbert space,
 let $C$ be a non-empty closed subset of $\HH$,  let $\phi\colon\left[0,\infty\right[ \to \left[0,\infty\right[ $,
let  $(x_n)_{n\in\NN}$ be a sequence of random vectors in $\HH$. Suppose that, 
for every $x\in C$, there exist non-negative summable sequences  of random variables
$(\zeta_n(x))_{n\in\NN}$ and $(t_n(x))_{n\in\NN}$
such that, for every $n\in\NN$,  $\zeta_n(x)$ and $t_n(x)$ are $\FF_n = \sigma(x_0,\ldots,x_n)$-measurable, and 
\begin{equation}\label{eq:rob}
(\forall n\in\NN)\quad \E[\phi(\|x_{n+1}-x\|)|\FF_n] \leq (1+t_n(x))\phi(\|x_{n}-x\|) + \zeta_n(x)
\quad \text{$\boldsymbol{\mathsf{P}}$-a.s}.
\end{equation}
Suppose that $\phi$ is  strictly increasing  and $\lim_{\xi\to\infty}\phi(\xi) = +\infty$.
Then the following hold.
\begin{enumerate}
\item\label{p:fejeri} $( \|x_{n}-x\|)_{n\in\NN}$ is bounded and converges  $\boldsymbol{\mathsf{P}}$-a.s.
\item \label{p:fejerii} 
There exists a subset $\boldsymbol{\Omega}^*$ with $\boldsymbol{\mathsf{P}}(\boldsymbol{\Omega}^*) =1$
such that for every $x\in C$ and every $\omega \in\boldsymbol{\Omega}^*$, $( \|x_{n}(\omega)-x\|)_{n\in\NN}$ converges .
\item \label{p:fejeriv}
  $(x_n)_{n\in\NN}$ converges weakly $\boldsymbol{\mathsf{P}}$-a.s. to a $C$-valued random vector 
if and only if  every its weak cluster point is in $C$ $\boldsymbol{\mathsf{P}}$-a.s.
\end{enumerate}
 \end{lemma}
\begin{remark} 
 A sequence $(x_n)_{n\in\NN}$ satisfying  \eqref{eq:rob} is called a stochastic $\phi$-quasi-Fej\'{e}r monotone with respect 
to the target set $C$. The connections of Lemma \ref{p:fejer} to existing work can be found in   \cite[Remark 2.4]{plc14}. 
\end{remark}
 
In view of the work in \cite[Theorem 3.3]{Guad2012}, we also have
  a variable metric extension of Lemma \ref{p:fejer}. 
\begin{proposition}
\label{p:11}
Let $\HH$ be a real separable Hilbert space,
 let $C$ be a non-empty closed subset of $\HH$,  let $\phi\colon\left[0,\infty\right[ \to \left[0,\infty\right[ $,
let $\alpha\in \left]0,\infty\right[$, let $W\in \BP_{\alpha}(\HH)$ and $(W_n)_{n\in\NN}$ be a 
 sequence  in $\BP_{\alpha}(\HH)$ 
such that $W_n\to W$ pointwise,
let  $(x_n)_{n\in\NN}$ be a sequence of random vectors in $\HH$. Suppose that, 
for every $x\in C$, there exist non-negative summable sequences  of random variables
$(\zeta_n(x))_{n\in\NN}$ and $(t_n(x))_{n\in\NN}$
such that, for every $n\in\NN$,  $\zeta_n(x)$ and $t_n(x)$ are $\FF_n = \sigma(x_0,\ldots,x_n)$-measurable, and 
\begin{equation}
 \label{Ve:vmqf1}
(\forall n\in\NN)\quad \E[\phi(\|x_{n+1}-x\|_{W_{n+1}})|\FF_n] \leq (1+t_n(x))\phi(\|x_{n}-x\|_{W_n}) + \zeta_n(x)
\quad \text{$\boldsymbol{\mathsf{P}}$-a.s}.
\end{equation}
Suppose that $\phi$ is  strictly increasing  and $\lim_{\xi\to\infty}\phi(\xi) = +\infty$.
Then the following hold.
\begin{enumerate}
\item\label{p11:fejeri} $( \|x_{n}-x\|_{W_n})_{n\in\NN}$ is bounded and converges $\boldsymbol{\mathsf{P}}$-a.s.
\item \label{p11:fejerii} 
There exists a subset $\boldsymbol{\Omega}^*$ with $\boldsymbol{\mathsf{P}}(\boldsymbol{\Omega}^*) =1$
such that for every $x\in C$ and every $\omega \in\boldsymbol{\Omega}^*$, $( \|x_{n}(\omega)-x\|_{W_n})_{n\in\NN}$ converges .
\item \label{p11:fejeriv}
  $(x_n)_{n\in\NN}$ converges weakly $\boldsymbol{\mathsf{P}}$-a.s. to a $C$-valued random vector 
if and only if  every its weak cluster point is in $C$ $\boldsymbol{\mathsf{P}}$-a.s.
\end{enumerate}
\end{proposition}
\begin{proof} 
\ref{p11:fejeri}: Set $(\forall n\in\NN)$ $\xi_n=\|x_n-z\|_{W_n}$.
It follows from \eqref{Ve:vmqf1} and Lemma \ref{l:rob85} that 
$(\phi(\xi_n))_{n\in\NN}$ converges $\boldsymbol{\mathsf{P}}$-a.s., say 
$\phi(\xi_n)\to\lambda$. In turn, since 
$\lim_{t\to\pinf}\phi(t)=\pinf$, $(\xi_n)_{n\in\NN}$ is bounded $\boldsymbol{\mathsf{P}}$-a.s. 
Let  $\omega \in \boldsymbol{\Omega}$ such that $(\xi_n(\omega))_{n\in\NN}$ is bounded 
and, to show that it converges, it suffices to show that it cannot
have two distinct cluster points. Suppose to the contrary that
we can extract two subsequences $(\xi_{k_n}(\omega))_{n\in\NN}$ and 
$(\xi_{l_n})_{n\in\NN}(\omega)$ such that $\xi_{k_n}(\omega)\to\eta(\omega)$ and 
$\xi_{l_n}(\omega)\to\zeta(\omega)>\eta(\omega)$, and fix 
$\varepsilon\in\left]0,(\zeta-\eta)/2\right[$. Then, for $n$
sufficiently large, $\xi_{k_n}(\omega)\leq\eta(\omega)+\varepsilon<
\zeta(\omega)-\varepsilon\leq\xi_{l_n}(\omega)$ and, since $\phi$ is strictly
increasing, $\phi(\xi_{k_n}(\omega))\leq\phi(\eta(\omega)+\varepsilon)<
\phi(\zeta(\omega)-\varepsilon)\leq\phi(\xi_{l_n}(\omega))$. Taking the limit as
$n\to\pinf$ yields $\lambda(\omega)\leq\phi(\eta(\omega)+\varepsilon)<
\phi(\zeta(\omega)-\varepsilon)\leq\lambda(\omega)$, which is impossible.

\ref{p11:fejerii}: Since $\HH$ is separable, so is $C$ and hence there exists 
 a countable  subset $X$ of $C$ such that $\overline{X} = C $. In view of 
 \ref{p11:fejeri}, for each $x\in X$, there exists a subset $ \boldsymbol{\Omega}_x$ with probability 
$1$ such that $( \|x_{n}(\omega)-x\|_{W_n})_{n\in\NN}$ converges for every $\omega \in \boldsymbol{\Omega}_x$.
Define $\boldsymbol{\Omega}^* = \bigcap_{x\in X}\boldsymbol{\Omega}_x $. Since $X$ is countable,
$\boldsymbol{\mathsf{P}}(\boldsymbol{\Omega}^*) =1$.
 Now, let $x_0\in C$ and $\omega_0\in\boldsymbol{\Omega}^*$. Then, there exists 
a sequence $(c_k)_{k\in\NN}$ in $X$ such that $c_k\to x_0$. 
By \ref{p11:fejeri}, we have 
\begin{equation}
(\forall k\in\NN)(\exists \tau_k\colon \boldsymbol{\Omega}\to \left[0,+\infty\right[)
(\forall \omega \in\boldsymbol{\Omega}_{c_k} )\quad \|x_{n}(\omega)-c_k\|_{W_n}\to\tau_k(\omega).
\end{equation}
Moreover, set $\mu = \sup_{n\in\NN}\|W_n\|$. Then $\mu < +\infty$ by  Banach-Steinhaus Theorem.
Then, for every $n\in\NN$ and $k\in\NN$, we have
\begin{alignat}{2}
 -\sqrt{\mu} \|c_k-x_0\|\leq
-\|c_k-x_0\|_{W_n} \leq  \|x_{n}(\omega_0)-x_0\|_{W_n} -\|x_{n}(\omega_0)-c_k\|_{W_n} &\leq \|c_k-x_0\|_{W_n}\notag\\
&\leq \sqrt{\mu} \|c_k-x_0\|.
\end{alignat}
Therefore,
\begin{alignat}{2}
(\forall k\in\NN)\quad -\sqrt{\mu} \|c_k-x_0\|&\leq \varliminf_{n\to\infty}  \|x_{n}(\omega_0)-x_0\|_{W_n}
 - \lim_{n\to\infty}\|x_{n}(\omega_0)-c_k\|_{W_n}\notag\\
&= \varliminf_{n\to\infty}  \|x_{n}(\omega_0)-x_0\|_{W_n}
 - \tau_k(\omega_0)\notag\\
&\leq \varlimsup_{n\to\infty}  \|x_{n}(\omega_0)-x_0\|_{W_n}
 - \tau_k(\omega_0)\notag\\
&\leq \sqrt{\mu} \|c_k-x_0\|.
\end{alignat}
Now, let $k\to\infty$, we get  $ \lim_{n\to\infty}  \|x_{n}(\omega_0)-x_0\|_{W_n}
 =\lim_{k\to\infty} \tau_k(\omega_0) $ which proves \ref{p11:fejerii}.

\ref{p11:fejeriv}:
Necessity is clear. To show sufficiency,  let $\Omega$ be the set of all $\omega$ such that every weak 
sequential cluster point of $(x_n(\omega))_{n\in\NN}$ is in $C$. Then $\Omega$ has probability $1$,
so is $\Omega_* = \Omega\cap \boldsymbol{\Omega}^*$. 
 Let $\omega\in \Omega_*$ and 
$x(\omega)$ and $y(\omega)$ be two weak cluster points of $(x_n(\omega))_{n\in\NN}$,
say $x_{k_n}(\omega)\weakly x(\omega)$ and $x_{l_n}(\omega)\weakly y(\omega)$. Then it follows from
 \ref{p11:fejerii} that $(\|x_n(\omega)-x(\omega)\|_{W_n})_{n\in\NN}$ 
and $(\|x_n(\omega)-y(\omega)\|_{W_n})_{n\in\NN}$ converge. Moreover,
$\|x(\omega)\|_{W_n}^2=\scal{W_nx(\omega)}{x(\omega)}\to\scal{Wx(\omega)}{x(\omega)}$ and, likewise,
$\|y(\omega)\|_{W_n}^2\to\scal{Wy(\omega)}{y(\omega)}$. Therefore, since
\begin{alignat}{2}
(\forall n\in\NN)\quad\scal{W_nx_n(\omega)}{x(\omega)-y(\omega)} 
&=\frac12\big(\|x_n(\omega)-y(\omega)\|_{W_n}^2-\|x_n(\omega)-x(\omega)\|_{W_n}^2\notag\\
&\quad+\|x(\omega)\|_{W_n}^2-\|y(\omega)\|_{W_n}^2\big),
\end{alignat}
the sequence $(\scal{W_nx_n(\omega)}{x(\omega)-y(\omega)})_{n\in\NN}$ converges, say 
$\scal{W_n x_n(\omega)}{x(\omega)-y(\omega)}\to\lambda(\omega)\in\RR$, which implies that
\begin{equation}
\label{Ve:elnido2012-03-07b}
\scal{x_n(\omega)}{W_n(x(\omega)-y(\omega))}\to\lambda(\omega)\in\RR.
\end{equation}
However, since $x_{k_n}(\omega)\weakly x(\omega)$ and $W_{k_n}(x(\omega)-y(\omega))\to W(x(\omega)-y(\omega))$, 
it follows from \eqref{Ve:elnido2012-03-07b} and 
\cite[Lemma~2.41(iii)]{livre1} that $\scal{x(\omega)}{W(x(\omega)-y(\omega))}=\lambda(\omega)$. 
Likewise, passing to the limit along the subsequence 
$(x_{l_n}(\omega))_{n\in\NN}$ in \eqref{Ve:elnido2012-03-07b} yields
$\scal{y(\omega)}{W(x(\omega)-y(\omega))}=\lambda$. Thus,
\begin{alignat}{2}
\label{Ve:elnido2012-03-08b}
0=\scal{x(\omega)}{W(x(\omega)-y(\omega))}-\scal{y(\omega)}{W(x(\omega)-y(\omega))} 
&=\scal{x(\omega)-y(\omega)}{W(x(\omega)-y(\omega))}\notag\\
&\geq\alpha\|x(\omega)-y(\omega)\|^2.
\end{alignat}
This shows that $x(\omega)=y(\omega)$. Upon invoking 
\ref{p11:fejerii} and
\cite[Lemma~2.38]{livre1}, 
we conclude that $x_n(\omega)\weakly x(\omega)$ and hence we obtain the conclusion.
\end{proof}

 \section{A stochastic forward-backward-forward splitting algorithm}
\label{s:sfbf}
\noindent
The forward-backward-forward splitting algorithm 
was firstly proposed in \cite{Tseng00} to solve inclusion involving
the sum of a maximally monotone operator and a Lipschitzian  monotone operator. 
In \cite{siop2}, it was revisited to include computational errors. 
Below, we extend it to a stochastic setting. 
 The following theorem is 
a stochastic version of \cite[Theorem 3.1]{Bang13}.

\begin{theorem}
\label{r:1}
Let $\KKK$ be a real  separable Hilbert space with the scalar product $\pscal{\cdot}{\cdot}$
and the associated norm $|||\cdot|||$.
Let $\alpha$ and $\beta$ be in $ \left]0,+\infty\right[$, 
let $(\eta_n)_{n\in\NN}$ be a sequence in $\ell_{+}^1(\NN)$,
and let $(\UU_n)_{n\in\NN}$ 
be a sequence in $\BL(\KKK)$ such that 
\begin{equation}
\label{e:tse}
\mu =\sup_{n\in\NN} \|\UU_n\| < +\infty\quad 
\text{and}\quad(\forall \xx\in\KKK)\; (1+\eta_n)\pscal{\xx}{\UU_{n+1}\xx} \geq \pscal{\xx}{\UU_{n}\xx} \geq \alpha |||\xx|||^2.
\end{equation}
Let $\AAA\colon\KKK\to 2^{\KKK}$ be maximally monotone,
let $\BB\colon\KKK\to \KKK$ be a monotone and $\beta$-Lipschitzian operator on
$\KKK$ such that $\zer(\AAA+\BB) \not= \emp$.
Let $(\aaa_n)_{n\in\NN}$, $(\bb_n)_{n\in\NN}$, and $(\cc_n)_{n\in\NN}$ 
be  sequences of  square integrable $\KKK$-valued random vectors.
Let $\xx_0$ be a square integrable $ \KKK$-valued random vector,
 let $\varepsilon \in \left]0, 1/(\beta\mu +1)\right[$,
let $(\gamma_n)_{n\in\NN}$ be a sequence 
in $\left[\varepsilon, (1-\varepsilon)/(\beta\mu)\right]$, and set
\begin{equation}
\label{e:Tsengaa}
(\forall n\in\NN)\quad
\begin{array}{l}
\left\lfloor
\begin{array}{l}
\yy_n= \xx_n-\gamma_n\UU_n(\BB \xx_n+\aaa_n)\\[1mm]
\pp_n = J_{\gamma_n \UU_n \AAA}\yy_n + \bb_n\\
\qq_n = \pp_n -\gamma_n \UU_n(\BB\pp_n+\cc_n)\\
\xx_{n+1} = \xx_n -\yy_n +\qq_n.
\end{array}
\right.\\[2mm]
\end{array}
\end{equation}
 Suppose that 
$\big(\sqrt{\E[|||\boldsymbol{a}_n |||^2|\FF_n]}\big)_{n\in\NN}$, $\big(\sqrt{\E[|||\boldsymbol{b}_n |||^2|\FF_n]}\big)_{n\in\NN}$,
 and $\big(\sqrt{\E[|||\boldsymbol{c}_n |||^2|\FF_n]}\big)_{n\in\NN}$ 
are  summable \text{$\boldsymbol{\mathsf{P}}$-a.s.},
the following hold for some  $\zer(\AAA+\BB)$-valued random vector $\overline{\xx}$.
\begin{enumerate}
 \item
\label{Tsengia} $\sum_{n\in\NN}\E[|||\xx_n-\pp_n|||^2|\FF_n] < \pinf$ and 
$\sum_{n\in\NN}\E[|||\yy_n-\qq_n|||^2| \FF_n] < \pinf$ $\boldsymbol{\mathsf{P}}$-a.s.
\item \label{Tsengiia}
$\xx_n \weakly \overline{\xx}$ and 
$J_{\gamma_n\UU_n\AAA}(\xx_n-\gamma_n\UU_n\BB\xx_n) \weakly \overline{\xx}$ $\boldsymbol{\mathsf{P}}$-a.s.
\item\label{Tsengiii} Suppose that one of the following is satisfied for some subset 
$\widetilde{\boldsymbol{\Omega}}\subset \boldsymbol{\Omega} $
 with $\boldsymbol{\mathsf{P}}( \widetilde{\boldsymbol{\Omega}}) =1$.
\begin{enumerate}
 \item\label{Tsengiiiba} $\AAA + \BB$ is demiregular 
(see \cite[Definition 2.3]{plc2010}) at $\overline{\xx}(\omega)$ for every $\omega\in\widetilde{\boldsymbol{\Omega}} $.
\item \label{Tsengiiica}$\AAA$ or $\BBB$ is uniformly monotone at $\overline{\xx}(\omega)$ 
for every $\omega\in\widetilde{\boldsymbol{\Omega}}$.
\end{enumerate}
Then $\xx_n\to\overline{\xx}$ and $J_{\gamma_n\UU_n\AAA}(\xx_n-\gamma_n\UU_n\BB\xx_n) \to\overline{\xx}$ $\boldsymbol{\mathsf{P}}$-a.s.
\end{enumerate}
\end{theorem}
\begin{proof}
It follows from \cite[Lemma 3.7]{Varm12} that the
sequences $(\xx_n)_{n\in\NN}$, $(\yy_n)_{n\in\NN}$, 
$(\pp_n)_{n\in\NN}$ and $(\qq_n)_{n\in\NN}$ are well defined. 
Moreover, using \cite[Lemma~2.1(i)(ii)]{Guad2012} and \eqref{e:tse},
for every sequence of random vectors $\KKK$-valued $(\boldsymbol{z}_n)_{n\in\NN}$,
we have
\begin{equation}
\label{e:absol}
 \sum_{n\in\NN}\sqrt{\E[|||\boldsymbol{z}_n |||^2|\FF_n]} < +\infty\quad \text{$\boldsymbol{\mathsf{P}}$-a.s.}
\quad \Leftrightarrow\quad 
\sum_{n\in\NN}\sqrt{\E[|||\boldsymbol{z}_n|||^{2}_{\UU_{n}^{-1}}|\FF_n]} < +\infty \quad \text{$\boldsymbol{\mathsf{P}}$-a.s.}
\end{equation}
and 
\begin{equation}
\label{e:absolun}
\sum_{n\in\NN}\sqrt{\E[|||\boldsymbol{z}_n |||^2|\FF_n]} < +\infty\quad \text{$\boldsymbol{\mathsf{P}}$-a.s.}
\quad \Leftrightarrow\quad  
\sum_{n\in\NN}\sqrt{\E[|||\boldsymbol{z}_n|||^{2}_{\UU_{n}}|\FF_n]} < +\infty\quad \text{$\boldsymbol{\mathsf{P}}$-a.s.}
\end{equation}
Let us set, for every $n\in\NN$,
\begin{equation}
\label{e:Tseng}
 \begin{cases}
 \widetilde{\yy}_n= \xx_n-\gamma_n \UU_n\BB\xx_n\\
\widetilde{\pp}_n = J_{\gamma_n \UU_n \AAA} \widetilde{\yy}_n\\
\widetilde{\qq}_n = \widetilde{\pp}_n -\gamma_n \UU_n\BB\widetilde{\pp}_n\\
\widetilde{\xx}_{n+1} = \xx_n - \widetilde{\yy}_n +\widetilde{\qq}_n,
\end{cases}
\;\; \text{and} \quad
\begin{cases}
\uu_n = \gamma^{-1}_n\UU^{-1}_n(\xx_n-\widetilde{\pp}_n) + \BB\widetilde{\pp}_n-\BB\xx_n\\
\ee_n = \widetilde{\xx}_{n+1} - \xx_{n+1}\\
\dd_n =\qq_n-\widetilde{\qq}_n +  \widetilde{\yy}_n-\yy_{n}.
\end{cases}
\end{equation}
Then \eqref{e:Tseng} yields 
\begin{equation}
 (\forall n\in\NN)\quad \uu_n = \gamma^{-1}_n\UU^{-1}_n(\ \widetilde{\yy}_n-\widetilde{\pp}_n) 
+ \BB\widetilde{\pp}_n \in \AAA\widetilde{\pp}_n+\BB\widetilde{\pp}_n,
\end{equation}
and \eqref{e:Tseng}, Lemma \cite[Lemma 3.7(ii)]{Varm12},
and the Lipschitzianity of $\BB$ on $\KKK$ yield 
\begin{equation}
\label{e:abs1}
(\forall n\in\NN)\quad
\begin{cases}
|||\yy_n -\widetilde{\yy}_n|||_{\UU_{n}^{-1}}& \leq  
(\beta\mu)^{-1}|||\aaa_n |||_{\UU_{n}}\\
 |||\pp_n -\widetilde{\pp}_n |||_{\UU_{n}^{-1}}& \leq   
|||\bb_n |||_{\UU_{n}^{-1}} + (\beta\mu)^{-1} |||\aaa_n |||_{\UU_{n}}\\
 |||\qq_n -\widetilde{\qq}_n |||_{\UU_{n}^{-1}} &\leq 
2\big( |||\bb_n |||_{\UU_{n}^{-1}} + (\beta\mu)^{-1} |||\aaa_n |||_{\UU_{n}}\big)+ (\beta\mu)^{-1} |||\cc_n |||_{\UU_{n}}.
\end{cases}
\end{equation}
Since $\big(\sqrt{\E[|||\boldsymbol{a}_n |||^2|\FF_n]}\big)_{n\in\NN}$, $\big(\sqrt{\E[|||\boldsymbol{b}_n |||^2|\FF_n]}\big)_{n\in\NN}$,
 and $\big(\sqrt{\E[|||\boldsymbol{c}_n |||^2|\FF_n]}\big)_{n\in\NN}$ 
are  summable \text{$\boldsymbol{\mathsf{P}}$-a.s.}, using Jensen's inequality, we derive from
\eqref{e:absol}, \eqref{e:absolun}, \eqref{e:Tseng}, and \eqref{e:abs1} that
\begin{alignat}{2}
 \label{e:pabsl}
\begin{cases}
 \sum_{n\in\NN}\E[|||\pp_n -\widetilde{\pp}_n ||| |\FF_n] < +\infty
\quad \text{and} \quad  \sum_{n\in\NN} \E[|||\pp_n -\widetilde{\pp}_n |||_{\UU_{n}^{-1}}|\FF_n] < +\infty\quad\text{$\boldsymbol{\mathsf{P}}$-a.s.} \\
 \sum_{n\in\NN}\E[|||\qq_n -\widetilde{\qq}_n ||| |\FF_n] < +\infty
\quad \text{and} \quad  \sum_{n\in\NN}\E[|||\qq_n -\widetilde{\qq}_n |||_{\UU_{n}^{-1}}|\FF_n]   < +\infty\quad\text{$\boldsymbol{\mathsf{P}}$-a.s.} \\
 \sum_{n\in\NN}\E[|||\dd_n ||| |\FF_n] < +\infty
\quad \text{and} \quad  \sum_{n\in\NN} \E[|||\dd_n |||_{\UU_{n}^{-1}} |\FF_n]  < +\infty\quad \text{$\boldsymbol{\mathsf{P}}$-a.s.}
\end{cases}
\end{alignat}
Noting that 
\begin{alignat}{2}
\label{eq:aa1}
2\E[ ||| \yy_n -\widetilde{\yy}_n|||_{\UU_{n}^{-1}}^2 |\FF_n]
 \leq 2 (\beta\mu)^{-2}\E[|||\aaa_n |||^{2}_{\UU_{n}}|\FF_n ],
\end{alignat}
and 
\begin{alignat}{2}
\label{eq:aa2}
2\E[ ||| \qq_n -\widetilde{\qq}_n|||_{\UU_{n}^{-1}}^2 |\FF_n] \leq 
24\big( \E[ |||\bb_n |||^{2}_{\UU_{n}^{-1}}|\FF_n] +(\beta\mu)^{-2} \E[ |||\aaa_n |||^{2}_{\UU_{n}}+|||\cc_n |||^{2}_{\UU_{n}}|\FF_n]  \big)
\end{alignat}
Therefore, upon setting $c = \max\{26(\beta\mu)^{-2}, 24\}$,
and adding \eqref{eq:aa1} and \eqref{eq:aa2}, we get 
 \begin{alignat}{2}
\label{eq:aa3}
2\E[ ||| \yy_n -\widetilde{\yy}_n|||_{\UU_{n}^{-1}}^2 |\FF_n]
 +2\E[ ||| \qq_n -\widetilde{\qq}_n |||_{\UU_{n}^{-1}}^2 |\FF_n]
 &\leq c \big(\E[ ||| \aaa_n |||_{\UU_{n}}^2 |\FF_n]
 +\E[ ||| \bb_n |||_{\UU_{n}^{-1}}^2 |\FF_n] \notag\\
&\hspace{3cm}+\E[ ||| \cc_n |||_{\UU_{n}}^2 |\FF_n]\big).
\end{alignat}
Now, using \eqref{eq:aa3}, \eqref{eq:cond1}, \eqref{eq:cond2},
\eqref{e:absol}, \eqref{e:absolun} and \eqref{e:Tseng},
 we have
 \begin{alignat}{2}
 \label{e:2c}
 \sum_{n\in\NN}\E[ |||\dd_n |||_{\UU_{n}^{-1}}^2 |\FF_n] &\leq 
2\sum_{n\in\NN}\E[ ||| \yy_n -\widetilde{\yy}_n|||_{\UU_{n}^{-1}}^2 |\FF_n]
 +2\sum_{n\in\NN}\E[ ||| \qq_n -\widetilde{\qq}_n |||_{\UU_{n}^{-1}}^2 |\FF_n]\notag \\
 &\leq c \bigg(\sum_{n\in\NN}\E[ ||| \aaa_n |||_{\UU_{n}}^2 |\FF_n]
 +\sum_{n\in\NN}\E[ ||| \bb_n |||_{\UU_{n}^{-1}}^2 |\FF_n] +\sum_{n\in\NN}\E[ ||| \cc_n |||_{\UU_{n}}^2 |\FF_n]\bigg)\notag\\
 &\leq c\tau_0\bigg(\sum_{n\in\NN}\sqrt{\E[ ||| \aaa_n |||_{\UU_n}^2 |\FF_n]}
 +\sum_{n\in\NN}\sqrt{\E[ ||| \bb_n |||_{\UU_{n}^{-1}}^2 |\FF_n] }\notag\\
&\hspace{7cm} + \sum_{n\in\NN}\sqrt{\E[ ||| \cc_n |||_{\UU_n}^2 |\FF_n] }\bigg)\notag\\
 &<+\infty\quad \text{$\boldsymbol{\mathsf{P}}$-a.s., }
 \end{alignat}
 where we define 
 \begin{equation}
 \tau_0 =\sup_{n\in\NN}\bigg\{\sqrt{\E[ ||| \aaa_n |||_{\UU_n}^2 |\FF_n]}, 
\sqrt{\E[ ||| \bb_n |||_{\UU^{-1}_n}^2 |\FF_n]}, \sqrt{\E[ ||| \cc_n |||_{\UU_n}^2 |\FF_n]} \bigg\}
  < +\infty\quad \text{$\boldsymbol{\mathsf{P}}$-a.s.}
 \end{equation}
Now, let $\xx\in \zer(\AAA+\BB)$. Then, for every $n\in\NN$,
 $(\xx,-\gamma_n \UU_n\BB\xx) \in\gra (\gamma_n\UU_n\AAA)$
and \eqref{e:Tseng} yields 
$(\widetilde{\pp}_n,\widetilde{\yy}_n-\widetilde{\pp}_n) \in \gra(\gamma_n \UU_n\AAA)$. 
Hence, by  monotonicity of $\UU_n\AAA$ with respect to the scalar product 
$\pscal{\cdot}{\cdot}_{\UU_{n}^{-1}}$, 
we have  
$
\pscal{\widetilde{\pp}_n-\xx}{\widetilde{\pp}_n-\widetilde{\yy}_n-\gamma_n\UU_n\BB\xx}_{\UU_{n}^{-1}} \leq 0.
$
Moreover, by  monotonicity of $\UU_n\BB$ with respect to 
the scalar product $\pscal{\cdot}{\cdot}_{\UU_{n}^{-1}}$, we also have 
$
\pscal{\widetilde{\pp}_n-\xx}{\gamma_n \UU_n \BB\xx -\gamma_n \UU_n \BB\widetilde{\pp}_n }_{\UU_{n}^{-1}}\leq 0.
$
By adding the last two inequalities, we obtain 
\begin{equation}
(\forall n\in\NN)\quad
\pscal{\widetilde{\pp}_n-\xx}{\widetilde{\pp}_n-\widetilde{\yy}_n-\gamma_n \UU_n\BB\widetilde{\pp}_n}_{\UU_{n}^{-1}} \leq 0.
\end{equation}
In turn, we derive from \eqref{e:Tseng} that 
\begin{alignat}{2}
\label{e:tse1}
(\forall n\in\NN)\quad
 &2\gamma_n \pscal{\widetilde{\pp}_n-\xx}{ \UU_n\BB\xx_n-\UU_n\BB\widetilde{\pp}_n}_{\UU_{n}^{-1}}\notag\\
&= 2\pscal{\widetilde{\pp}_n-\xx}{\widetilde{\pp}_n
-\widetilde{\yy}_n -\gamma_n \UU_n \BB\widetilde{\pp}_n}_{\UU_{n}^{-1}}
\notag\\
&\quad +2 \pscal{\widetilde{\pp}_n-\xx}{\gamma_n \UU_n \BB\xx_n
 +\widetilde{\yy}_n-\widetilde{\pp}_n }_{\UU_{n}^{-1}}\notag\\ 
& \leq 2 \pscal{\widetilde{\pp}_n-\xx}{\gamma_n \UU_n \BB\xx_n 
+\widetilde{\yy}_n-\widetilde{\pp}_n }_{\UU_{n}^{-1}}\notag\\
&= 2 \pscal{\widetilde{\pp}_n-\xx}{\xx_n-\widetilde{\pp}_n }_{\UU_{n}^{-1}}\notag\\
& = |||\xx_n-\xx|||_{\UU^{-1}_n}^2
-|||\widetilde{\pp}_n-\xx|||_{\UU_{n}^{-1}}^2
-|||\xx_n-\widetilde{\pp}_n|||_{\UU_{n}^{-1}}^2.
\end{alignat}
Hence, using \eqref{e:Tseng}, 
\eqref{e:tse1}, the $\beta$-Lipschitz continuity of $\BB$, and 
\cite[Lemma~2.1(ii)]{Guad2012}, for every $n\in\NN$, we obtain 
\begin{alignat}{2}
\label{e:troicho}
|||\widetilde{\xx}_{n+1}-\xx|||_{\UU_{n}^{-1}}^2 &= 
|||\widetilde{\qq}_n + \xx_n -\widetilde{\yy}_n -\xx|||_{\UU_{n}^{-1}}^2\notag\\
& = |||(\widetilde{\pp}_n-\xx) + \gamma_n \UU_n(\BB\xx_n-\BB\widetilde{\pp}_n)|||_{\UU_{n}^{-1}}^2\notag\\
&= |||\widetilde{\pp}_n-\xx|||_{\UU_{n}^{-1}}^2 +2\gamma_n\pscal{\widetilde{\pp}_n-\xx}{\BB\xx_n-\BB\widetilde{\pp}_n}\notag\\
&\quad \hspace{4cm}+ \gamma_{n}^2 |||\UU_n (\BB\xx_n-\BB\widetilde{\pp}_n)|||_{\UU_{n}^{-1}}^2\notag\\
&\leq |||\xx_n-\xx|||_{\UU^{-1}_n}^2 -|||\xx_n-\widetilde{\pp}_n|||_{\UU_{n}^{-1}}^2 \notag\\
&\quad\hspace{4cm} +\gamma_{n}^2\mu\beta^2|||\xx_n-\widetilde{\pp}_n|||^2\notag\\  
&\leq |||\xx_n-\xx|||_{\UU^{-1}_n}^2 - \mu^{-1}|||\xx_n-\widetilde{\pp}_n|||^2 \notag\\
&\quad\hspace{4cm} +\gamma_{n}^2\mu\beta^2|||\xx_n-\widetilde{\pp}_n|||^2.
\end{alignat}
Hence,     
it follows from \eqref{e:tse} and \cite[Lemma~2.1(i)]{Guad2012}
that
\begin{alignat}{2}
\label{e:t21ab}
(\forall n\in\NN)\quad
 |||\widetilde{\xx}_{n+1}-\xx|||_{\UU_{n+1}^{-1}}^2 &\leq 
(1+\eta_n)|||\xx_n-\xx|||_{\UU^{-1}_n}^2\notag\\ 
&\quad-\mu^{-1}(1-\gamma^{2}_n \beta^2\mu^2)|||\xx_n-\widetilde{\pp}_n|||^2.
\end{alignat}
Consequently, 
\begin{equation}
\label{e:t21a}
(\forall n\in\NN)\quad
 |||\widetilde{\xx}_{n+1}-\xx|||_{\UU_{n+1}^{-1}} \leq (1+\eta_n)|||\xx_n-\xx|||_{\UU^{-1}_n}.
\end{equation}
For every $n\in\NN$, set 
\begin{equation} 
\varepsilon_n = \sqrt{\mu\alpha^{-1}}\Big(
2\big( |||\bb_n |||_{\UU_{n}^{-1}} + (\beta\mu)^{-1} |||\aaa_n |||_{\UU_{n}}\big)  
+ (\beta\mu)^{-1} |||\cc_n |||_{\UU_{n}} + (\beta\mu)^{-1}|||\aaa_n |||_{\UU_{n}}\Big) .
\end{equation}
Then $( \E[\varepsilon_n|\FF_n])_{n\in\NN}$ is summable \text{$\boldsymbol{\mathsf{P}}$-a.s.}  
by \eqref{e:absol} and we derive 
from \cite[Lemma~2.1(ii)(iii)]{Guad2012}, and \eqref{e:pabsl} that 
\begin{alignat}{2}
(\forall n\in\NN)\quad |||\ee_n|||_{\UU_{n+1}^{-1}}& 
= |||\widetilde{\xx}_{n+1}-\xx_{n+1}|||_{\UU_{n+1}^{-1}}\notag\\
&\leq \sqrt{\alpha^{-1}}|||\widetilde{\xx}_{n+1}-\xx_{n+1}|||\notag\\
&\leq \sqrt{\mu\alpha^{-1}}|||\widetilde{\xx}_{n+1}-\xx_{n+1}|||_{\UU_{n}^{-1}}\notag\\
&\leq  \sqrt{\mu\alpha^{-1}}
(|||\widetilde{\yy}_n-\yy_{n}|||_{\UU_{n}^{-1}} + |||\widetilde{\qq}_{n}-\qq_{n}|||_{\UU_{n}^{-1}})\notag\\
&\leq \varepsilon_n.
\end{alignat}
In turn, we derive from \eqref{e:t21a} that  
\begin{alignat}{2}
\label{e:oac}
(\forall n\in\NN)\quad
|||\xx_{n+1}-\xx|||_{\UU_{n+1}^{-1}} &\leq 
|||\widetilde{\xx}_{n+1}-\xx|||_{\UU_{n+1}^{-1}}+  
|||\widetilde{\xx}_{n+1}-\xx_{n+1}|||_{\UU_{n+1}^{-1}}\notag\\
&\leq|||\widetilde{\xx}_{n+1}-\xx|||_{\UU_{n+1}^{-1}}+ \varepsilon_n\notag\\
&\leq (1+\eta_n)|||\xx_n-\xx|||_{\UU^{-1}_n}+\varepsilon_n.
\end{alignat}
By assumption,  since $\E[\|\xx_0\|^2] $ is finite, by induction, for every $n\in\NN$, $\E[\|\xx_n\|^2]$ is finite 
 and hence $\E[\|\xx_n\|]$ and $\E[\|\xx_n\|_{\UU^{-1}_n}]$ are finite too. 
By taking the conditional expectation with respect to $\FF_n$ and note that $ |||\xx_n-\xx|||_{\UU^{-1}_n}$ is $\FF_n$-measurable,
we obtain
\begin{equation}
(\forall n\in\NN)\quad
\E[|||\xx_{n+1}-\xx|||_{\UU_{n+1}^{-1}}|\FF_n]  \leq (1+\eta_n)|||\xx_n-\xx|||_{\UU^{-1}_n}
+ \E[\varepsilon_n|\FF_n].
\end{equation}
This shows that $(\xx_n)_{n\in\NN}$ is $|\cdot|$--quasi-Fej\'er monotone with respect to the target set
$\zer(\AAA+\BB)$ relative to $(\UU^{-1}_n)_{n\in\NN}$. Moreover, 
 $(|||\xx_n-\xx|||_{\UU^{-1}_n})_{n\in\NN}$ is bounded. 
In turn, since $\BB$ and $(J_{\gamma_n \UU_n\AAA})_{n\in\NN}$
are Lipschitzian, and 
$(\forall n\in\NN)\; \xx = J_{\gamma_n\UU_n\AAA}(\xx -\gamma_n\UU_n\BB\xx)$,
we deduce from \eqref{e:Tseng} that 
$(\widetilde{\yy}_n)_{n\in\NN},(\widetilde{\pp}_{n})_{n\in\NN}$, 
and $(\widetilde{\qq}_{n})_{n\in\NN}$ are bounded. 
Therefore, 
\begin{equation}
\label{e:2a}
 \tau = \sup_{n\in\NN}\{|||\xx_n -\widetilde{\yy}_n + \widetilde{\qq}_n -\xx |||_{\UU_{n}^{-1}},
|||\xx_n-\xx|||_{\UU^{-1}_n}\}
 < +\infty\quad \text{$\boldsymbol{\mathsf{P}}$-a.s.}
\end{equation}
Hence, using \eqref{e:Tseng}, Cauchy-Schwarz for the norms
$(|||\cdot|||_{\UU_{n}^{-1}})_{n\in\NN}$, and \eqref{e:troicho}, 
we get, for every $n\in\NN$,
\begin{alignat}{2}
|||\xx_{n+1}-\xx |||_{\UU_{n}^{-1}}^2 
&= |||\xx_n -\yy_n +\qq_n -\xx |||_{\UU_{n}^{-1}}^2\notag\\
&= |||\widetilde{\qq}_n+ \xx_n -\widetilde{\yy}_n -\xx + \dd_n |||_{\UU_{n}^{-1}}^2\notag\\
&\leq |||\widetilde{\qq}_n+ \xx_n -\widetilde{\yy}_n -\xx|||_{\UU^{-1}_n}^2 
+ 2|||\widetilde{\qq}_n+ \xx_n -\widetilde{\yy}_n -\xx |||_{\UU_{n}^{-1}} |||\dd_n|||_{\UU_{n}^{-1}}\notag\\
 &\quad+ |||\dd_n|||_{\UU_{n}^{-1}}^2\notag\\
&\leq |||\xx_n-\xx|||_{\UU^{-1}_n}^2 -
\mu^{-1}(1-\gamma^{2}_n \beta^2\mu^2)|||\xx_n-\widetilde{\pp}_n|||^2 + \varepsilon_{1,n},
\end{alignat}
where $(\forall n\in\NN)\; \varepsilon_{1,n} = 
 2 ||| \widetilde{\qq}_n+ \xx_n -\widetilde{\yy}_n -\xx|||_{\UU_{n}^{-1}} |||\dd_n|||_{\UU_{n}^{-1}} + |||\dd_n|||_{\UU_{n}^{-1}}^2$.
In turn, for every $n\in\NN$, by \eqref{e:tse} and \cite[Lemma 2.1(i)]{Guad2012},
\begin{alignat}{2}
|||\xx_{n+1}-\xx |||_{\UU_{n+1}^{-1}}^2 &\leq (1+\eta_n)|||\xx_{n+1}-\xx |||_{\UU_{n}^{-1}}^2 \notag\\
&\leq(1+\eta_n) |||\xx_n-\xx|||_{\UU^{-1}_n}^2 -
\mu^{-1}(1-\gamma^{2}_n \beta^2\mu^2)|||\xx_n-\widetilde{\pp}_n|||^2\notag \\
&\quad+ (1+\eta_n)\varepsilon_{1,n}.
\end{alignat}
 Since, $J_{\gamma_n\AAA}\circ(\Id -\gamma_n \BB)$ is continuous, 
$\widetilde{\pp}_n$ is $\FF_n$-measurable. 
In turn, for every $n\in\NN$, 
\begin{alignat}{2}
\E[|||\xx_{n+1}-\xx |||_{\UU_{n+1}^{-1}}^2|\FF_n] 
&\leq(1+\eta_n) |||\xx_n-\xx|||_{\UU^{-1}_n}^2 -
\mu^{-1}(1-\gamma^{2}_n \beta^2\mu^2)|||\xx_n-\widetilde{\pp}_n|||^2\notag \\
&\quad+ \E[(1+\eta_n)\varepsilon_{1,n}|\FF_n] .
\end{alignat}
Let us prove that
\begin{equation}
 \label{e2}
 \sum_{n\in\NN}\E[ \varepsilon_{1,n} | \FF_n] < +\infty 
 \quad \text{$\boldsymbol{\mathsf{P}}$-a.s.}
 \end{equation}
 Indeed, since 
 $\Id -\gamma_n \BB$  is continuous, $\widetilde{\yy}_n$ and $\widetilde{\qq}_n$
 are $\FF_n$-measurable. Therefore, $|||\xx_n -\widetilde{\yy}_n + \widetilde{\qq}_n -\xx |||_{\UU_{n}^{-1}}$ is $\FF_n$-measurable 
 and hence by \eqref{e:pabsl} and \eqref{e:2a}, we obtain
 \begin{alignat}{2}
 \sum_{n\in\NN}\E[|||\xx_n -\widetilde{\yy}_n + \widetilde{\qq}_n -\xx |||_{\UU_{n}^{-1}} |||\dd_n |||_{\UU_{n}^{-1}} |\FF_n] 
 &=  \sum_{n\in\NN}|||\xx_n -\widetilde{\yy}_n + \widetilde{\qq}_n -\xx |||_{\UU_{n}^{-1}} \E[|||\dd_n |||_{\UU_{n}^{-1}} |\FF_n]\notag \\
 &\leq \tau  \sum_{n\in\NN} \E[|||\dd_n |||_{\UU_{n}^{-1}} |\FF_n]\notag \\
 &< +\infty\quad \text{$\boldsymbol{\mathsf{P}}$-a.s.},
 \end{alignat}
 which and \eqref{e:2c} prove \eqref{e2}.
It follows from Lemma 2.1 that
\begin{equation}
\label{e:emyeu}
 \sum_{n\in\NN} |||\xx_n-\widetilde{\pp}_n|||^2 < +\infty\quad \text{$\boldsymbol{\mathsf{P}}$-a.s.}
\end{equation}
(i): 
It follows from \eqref{e:emyeu} and \eqref{e:pabsl} that
\begin{equation}
 \sum_{n\in\NN}\E[|||\xx_n-\pp_n|||^2|\FF_n]   
\leq 2\sum_{n\in\NN} |||\xx_n-\widetilde{\pp}_n|||^2 
+ 2\sum_{n\in\NN}\E[|||\pp_n-\widetilde{\pp}_n |||^2|\FF_n ] < +\infty \quad\text{$\boldsymbol{\mathsf{P}}$-a.s.}
\end{equation}
Furthermore, we derive from \eqref{e:pabsl}, \eqref{e:absol} and \eqref{e:2c} that 
\begin{alignat}{2}
 \sum_{n\in\NN}\E[|||\yy_n -\qq_n |||^2]  &= \sum_{n\in\NN}
 \E[|||\widetilde{\qq}_n-\widetilde{\yy}_n  +\dd_n |||^2|\FF_n] \notag\\
&= \sum_{n\in\NN}\E[|||\widetilde{\pp}_n -\xx_n 
+\gamma_n\UU_n(\BB\xx_n -\BB\widetilde{\pp}_n) +\dd_n |||^2|\FF_n]  \notag \\
&\leq 3\Big( \sum_{n\in\NN} |||\xx_n -\widetilde{\pp}_n|||^2 
+\E[|||\gamma_n\UU_n(\BB\xx_n -\BB\widetilde{\pp}_n) |||^2
 + |||\dd_n |||^2|\FF_n] \Big)\notag \\
&< +\infty\quad \text{$\boldsymbol{\mathsf{P}}$-a.s.}
\end{alignat}

\ref{Tsengiia}: Let $\boldsymbol{\Omega}_0$ be the set of all $\omega \in \boldsymbol{\Omega}$ 
such that $(\xx_n(\omega))_{n\in\NN}$ is bounded and \eqref{e:emyeu} is satisfied. We have 
$\boldsymbol{\mathsf{P}}(\boldsymbol{\Omega}_0) =1$. Fix $\omega \in \boldsymbol{\Omega}_0$.
Let $\xx(\omega)$ be a weak cluster point of $(\xx_n(\omega))_{n\in\NN}$. 
Then there exists a subsequence $(\xx_{k_n}(\omega))_{n\in\NN}$ that converges weakly to
$\xx(\omega)$. Therefore $\widetilde{\pp}_{k_n}(\omega)\weakly \xx(\omega)$ by \eqref{e:emyeu} and by 
the definition of  $\boldsymbol{\Omega}_0$.
Furthermore, it follows from \eqref{e:Tseng} that $\uu_{k_n}(\omega)\to 0$. 
Hence, since $(\forall n\in\NN )\;(\widetilde{\pp}_{k_n}(\omega),\uu_{k_n}(\omega))\in\gra(\AAA+\BB)$, 
we obtain, $\xx(\omega)\in \zer(\AAA+\BB)$ \cite[Proposition 20.33(ii)]{livre1}. 
Altogether, it follows Proposition \ref{p:11}
that $\xx_n\weakly \overline{\xx}$ and hence that
$\widetilde{\pp}_n\weakly \overline{\xx}$.

% \ref{Tsengiiia}: Since $\AAA$ and $\BB$ are maximally monotone and $\dom \BB = \KKK$,
% $\AAA+\BB$ is maximally monotone \cite[Corollary 24.4(i)]{livre1},  $\zer(\AAA+\BB)$ is therefore 
% closed \cite[Proposition~23.39]{livre1}.
% Hence, the claims follow
% from \ref{Tsengi}, \eqref{e:oac}, and Lemma \ref{p:fejer1} .
Now, let $\boldsymbol{\Omega}_1$ be the set of all $\omega\in \boldsymbol{\Omega}$ such that 
$\xx_n(\omega)\weakly\overline{\xx}(\omega)$ and $\widetilde{\pp}_n(\omega) \weakly\overline{\xx}(\omega)$,
and $\widetilde{\pp}_n(\omega) - \xx_n(\omega)\to 0$.
 Then $\boldsymbol{\mathsf{P}}(\boldsymbol{\Omega}_1) =1$ and hence 
$\boldsymbol{\mathsf{P}}(\boldsymbol{\Omega}_1\cap\widetilde{\boldsymbol{\Omega}}) =1$. 

\ref{Tsengiiiba}: Fix $\omega\in \boldsymbol{\Omega}_1\cap\widetilde{\boldsymbol{\Omega}}$.
Then $\xx_n(\omega)\weakly\overline{\xx}(\omega)$ and $\widetilde{\pp}_n(\omega) \weakly\overline{\xx}(\omega)$. 
Furthermore, it follows from \eqref{e:Tseng} that $\uu_{n}(\omega)\to 0$. 
Hence, since $(\forall n\in\NN)\;
(\widetilde{\pp}_{n}(\omega),\uu_{n}(\omega))\in\gra(\AAA+\BB)$ 
and since $\AAA+\BB$ is demiregular at $\overline{\xx}(\omega)$ by our assumption, 
by \cite[Definition 2.3]{plc2010}, $\widetilde{\pp}_n(\omega)\to\overline{\xx}(\omega)$, 
 and therefore  $\xx_{n}(\omega)\to\overline{\xx}(\omega)$.

\ref{Tsengiiica}: Fix $\omega\in \boldsymbol{\Omega}_1\cap\widetilde{\boldsymbol{\Omega}}$.
If $\AAA$ or $\BB$ is uniformly monotone at 
$\overline{\xx}(\omega)$, then $\AAA +\BB$ is 
uniformly monotone at $\overline{\xx}(\omega)$. 
Therefore, the result follows from \cite[Proposition 2.4(i)]{plc2010}.
\end{proof}

\begin{corollary}
\label{l:Tsenglemma}
Let $\KKK$ be a real separable Hilbert space with the scalar product $\pscal{\cdot}{\cdot}$
and the associated norm $|||\cdot|||$.
Let $\beta$ be in $ \left]0,+\infty\right[$,
let $\AAA\colon\KKK\to 2^{\KKK}$ be maximally monotone,
let $\BB\colon\KKK\to \KKK$ be a monotone and $\beta$-Lipschitzian operator on
$\KKK$ such that $\zer(\AAA+\BB) \not= \emp$.
Let $(\aaa_n)_{n\in\NN}$, $(\bb_n)_{n\in\NN}$, and $(\cc_n)_{n\in\NN}$ 
be  sequences of square integrable $\KKK$-valued random vectors.
Let $\xx_0$ be a square integrable $\KKK$-valued random vector , let $\varepsilon \in \left]0, 1/(\beta+1)\right[$,
let $(\gamma_n)_{n\in\NN}$ be a sequence 
in $\left[\varepsilon, (1-\varepsilon)/\beta\right]$, and set
\begin{equation}
\label{e:Tsenga}
(\forall n\in\NN)\quad
\begin{array}{l}
\left\lfloor
\begin{array}{l}
\yy_n= \xx_n-\gamma_n(\BB \xx_n+\aaa_n)\\[1mm]
\pp_n = J_{\gamma_n \AAA}\yy_n + \bb_n\\
\qq_n = \pp_n -\gamma_n(\BB\pp_n+\cc_n)\\
\xx_{n+1} = \xx_n -\yy_n +\qq_n.
\end{array}
\right.\\[2mm]
\end{array}
\end{equation}
Suppose that the following conditions are satisfied with $\FF_n = \sigma(\xx_0,\ldots,\xx_n)$,  
\begin{equation}\label{eq:cond1}
\sum_{n\in\NN}\sqrt{\E[|||\aaa_n|||^2|\FF_n]} < \infty, \quad 
\sum_{n\in\NN}\sqrt{\E[|||\bb_n|||^2|\FF_n]} < \infty \quad 
\end{equation}
and 
\begin{equation}\label{eq:cond2}
\sum_{n\in\NN}\sqrt{\E[|||\cc_n|||^2|\FF_n]} < \infty \quad
\text{$\boldsymbol{\mathsf{P}}$-a.s.}
\end{equation}
Then the following hold for some  $\zer(\AAA+\BB)$-valued random vector $\overline{\xx}$.
\begin{enumerate}
 \item
\label{Tsengi} $\sum_{n\in\NN}\E[|||\xx_n-\pp_n|||^2|\FF_n] < \pinf$ and 
$\sum_{n\in\NN}\E[|||\yy_n-\qq_n|||^2|\FF_n]< \pinf$ $\boldsymbol{\mathsf{P}}$-a.s.
\item \label{Tsengii}
$\xx_n \weakly \overline{\xx}$ and 
$ J_{\gamma_n \AAA}(\xx_n-\gamma_n\BB\xx_n) \weakly \overline{\xx}$ $\boldsymbol{\mathsf{P}}$-a.s.
\item\label{Tsengiii} Suppose that one of the following is satisfied for some subset 
$\widetilde{\boldsymbol{\Omega}}\subset \boldsymbol{\Omega} $ with $\boldsymbol{\mathsf{P}}( \widetilde{\boldsymbol{\Omega}}) =1$.
\begin{enumerate}
 \item\label{Tsengiiib} $\AAA + \BB$ is demiregular 
(see \cite[Definition 2.3]{plc2010}) at $\overline{\xx}(\omega)$ for every $\omega\in\widetilde{\boldsymbol{\Omega}} $.
\item \label{Tsengiiic}$\AAA$ or $\BBB$ is uniformly monotone at $\overline{\xx}(\omega)$ 
for every $\omega\in\widetilde{\boldsymbol{\Omega}}$.
\end{enumerate}
Then $\xx_n\to\overline{\xx}$ and $J_{\gamma_n \AAA}(\xx_n-\gamma_n\BB\xx_n) \to\overline{\xx}$ $\boldsymbol{\mathsf{P}}$-a.s.
\end{enumerate}
\end{corollary}

\begin{remark} Here are some remarks.
In the case when $\BB$ is a general multi-valued maximally monotone operator or  a cocoercive operator,
the almost sure convergence of the Douglas-Rachford or forward-backward are proved in \cite{plc14} under the same 
type of condition on the stochastic errors. Furthermore, in the case when $\BB$ is cocoercive and uniformly monotone, 
the almost sure convergence of the forward-backward splitting is also proved in \cite{LSB14} under different conditions 
on stepsize and stochastic errors. One of the early work concerns with Lipschitzian monotone operator was in \cite{Nem11}.
\end{remark}

\begin{example}
Let $\boldsymbol{f}\colon\KKK\to\left[-\infty,+\infty\right]$ 
be a proper lower semicontinuous convex function, let $\alpha\in\RPP$, let $\beta\in\RPP$, 
let $\BB\colon\KKK\to\KKK$ be a monotone and $\beta$-Lipschitzian operator.
Let $(\aaa_n)_{n\in\NN}$, $(\bb_n)_{n\in\NN}$, and $(\cc_n)_{n\in\NN}$ 
be  sequences of square integrable $\KKK$-valued random vectors such that 
\eqref{eq:cond1} and \eqref{eq:cond2} are satisfied.
 Furthermore, let $\xx_0$ be a square integrable $\KKK$-valued random vector,
let $\varepsilon\in\left]0,\min\{1,1/(\beta+1)\}\right[$,
let $(\gamma_n)_{n\in\NN}$ be a 
sequence in $[\varepsilon,(1-\varepsilon)/\beta]$.
Suppose that the 
variational inequality
\begin{equation}
\label{e:2012-11:10}
\text{find}\quad \xxx\in\KKK\quad\text{such that}\quad
(\forall \yy\in\KKK)\quad\scal{\xxx-\yy}{\BB\xxx}+ \boldsymbol{f}(\xxx)\leq \boldsymbol{f}(\yy)
\end{equation}
admits at least one solution and set
\begin{equation}
\label{e:Tsengavar}
(\forall n\in\NN)\quad
\begin{array}{l}
\left\lfloor
\begin{array}{l}
\yy_n= \xx_n-\gamma_n(\BB\xx_n+\aaa_n)\\[1mm]
\pp_n = \arg\underset{\xx\in\KKK}{\min}\big(\boldsymbol{f}(\xx)  
+ \frac{1}{2\gamma_n}|||\xx-\yy_n|||^2\big) +\bb_n\\
\qq_n = \pp_n -\gamma_n (\BB\pp_n+\cc_n)\\
\xx_{n+1} = \xx_n -\yy_n +\qq_n.
\end{array}
\right.\\[2mm]
\end{array}
\end{equation}
Then, for almost all $\omega\in \boldsymbol{\Omega}$, 
 $(\xx_n(\omega))_{n\in\NN}$ converges weakly to a solution 
$\xxx(\omega)$ to \eqref{e:2012-11:10}.
\end{example}
\begin{proof}
Set $\AAA=\partial \boldsymbol{f}$
in Corollary~\ref{l:Tsenglemma}\ref{Tsengii}.
\end{proof}
\begin{remark} 
Since $(\gamma_n)_{n\in\NN}$ is bounded away from $0$, we have 
\begin{equation}
\label{infff}
\sum_{n\in\NN}\gamma_n = +\infty 
\quad 
\text{and}\quad
\sum_{n\in\NN}\gamma^{2}_n = +\infty. 
\end{equation}
While, in the standard stochastic gradient method \cite{Polyak87}, we  often require 
\begin{equation}
\label{efffff}
\sum_{n\in\NN}\gamma_n = +\infty 
\quad 
\text{and}\quad
\sum_{n\in\NN}\gamma^{2}_n < +\infty. 
\end{equation}
 Under the condition \eqref{efffff},
 the conditions on the stochastic errors in the 
stochastic gradient method are  weaker than \eqref{eq:cond1}-\eqref{eq:cond2} 
(see also \cite[Assumption 2 and Eq (4) ]{Yousefian} for the case 
of the projected stochastic gradient method). 
\end{remark}
We end this section by noting that, 
 in the case when $\UU_n =\UU$, we obtain a  preconditioned version of \eqref{e:Tsenga}. 
Some other preconditioned algorithms can be found 
in \cite{JCP14,Xu08}. 

\section{Monotone inclusions involving Lipschitzian operators}
\label{s:app}
The applications of the forward-backward-forward splitting algorithm 
considered in \cite{siop2,plc6,Tseng00} can be extended to a stochastic 
setting using Theorem \ref{r:1}. As an illustration, we present a stochastic version 
of the algorithm proposed in \cite[Eq. (3.1)]{plc6}.
Recall that
the parallel sum of $A\colon\HH\to 2^{\HH}$ and $B\colon\HH\to 2^{\HH}$ is \cite{livre1} 
\begin{equation}
\label{e:parasum}
A\infconv B=(A^{-1}+ B^{-1})^{-1}.
\end{equation} 

\begin{problem}
\label{CCP}
Let $\HH$ be a real separable Hilbert space,
let $m$ be a strictly positive integer, let $z\in\HH$,
let $A\colon\HH\to 2^{\HH}$ be maximally monotone operator, 
let $C\colon\HH\to\HH$ be monotone and $\nu_0$-Lipschitzian for some
$\nu_0\in\left]0,+\infty\right[$.
For every $i\in\{1,\ldots, m\}$, let $\GG_i$ be a real separable Hilbert space, 
let $r_i \in \GG_i$, let $B_i\colon \GG_i\to2^{\GG_i}$ 
be maximally monotone operator, 
let $D_i\colon \GG_i\to2^{\GG_i}$ be monotone and
such that $D_{i}^{-1}$ is $\nu_i$-Lipschitzian for some
$\nu_i\in\left]0,+\infty\right[$,
and let $L_i\colon\HH \to\GG_i$ 
be a nonzero bounded linear operator. 
Suppose that
\begin{equation}
\label{e:fbfcond}
z\in\ran
\bigg(A+\sum_{i=1}^mL^{*}_i\big((B_i\;\vuo\; D_i)(L_i\cdot-r_i)\big) + C\bigg).
\end{equation}
The problem is to solve the primal inclusion
\begin{equation}\label{primal12:41}
z\in A\overline{x} +\sum_{i=1}^mL^{*}_i
\big((B_i\;\vuo\; D_i)(L_i\overline{x}-r_i)\big) + C\overline{x} ,
\end{equation}
and  the dual inclusion
\begin{equation}\label{dual12:41}
(\forall i\in\{1,\ldots,m\})\quad
r_i \in -L_i(A +C)^{-1}\bigg(z -\sum_{i = 1}^m L_{i}^*\overline{v}_i \bigg) 
 +B^{-1}_i \overline{v}_i+  D^{-1}_i\overline{v}_i .
\end{equation}
We denote by $\mathcal{P}$ and $\mathcal{D}$ be the set of solutions 
to \eqref{primal12:41} and \eqref{dual12:41}, respectively.
\end{problem}

As shown in \cite{plc6}, Problem \ref{CCP} covers 
a wide class of problems in nonlinear analysis and
convex  optimization problems.
However, the algorithm in \cite[Theorem 3.1]{plc6} is studied in the deterministic. 
The following result extends this result to a stochastic setting.

Let us define $\KKK=\HH\oplus\GG_1\oplus\cdots\oplus\GG_m$ the Hilbert direct 
sum of the Hilbert spaces $\HH$ and $(\GG_i)_{1\leq i\leq m}$,
the scalar product and the associated 
norm of $\KKK$ respectively defined by
\begin{equation}
\label{e:palawan-mai2008b-}
\pscal{\cdot}{\cdot}
\colon((x,\boldsymbol{v}),(y,\boldsymbol{w}))\mapsto
\scal{x}{y}+ \sum_{i=1}^m\scal{v_i}{w_i}
\;\text{and}\;|||\cdot|||\colon
(x,\boldsymbol{v})\mapsto\sqrt{\|x\|^2+\sum_{i=1}^m\|v_i\|^2},
\end{equation}
where $\vv = (v_1,\ldots,v_m)$ and $\ww = (w_1,\ldots,w_m)$ are generic elements in 
$\GG_1\oplus\cdots\oplus\GG_m$.

\begin{corollary}
\label{t:Sv2}
Let $(a_{1,n})_{n\in\NN}, (b_{1,n})_{n\in\NN}$, and $(c_{1,n})_{n\in\NN}$
be  sequences of square integrable $\HH$-valued random vectors, and for every $i\in\{1,\ldots,m\}$, 
let $(a_{2,i,n})_{n\in\NN}, (b_{2,i,n})_{n\in\NN}$, and $(c_{2,i,n})_{n\in\NN}$ 
be  sequences of square integrable $\GG_i$-valued random vectors. Furthermore, set 
\begin{equation}
 \beta = \max\{\nu_0,\nu_1,\ldots, \nu_m\} + \sqrt{\sum_{i=1}^m\|L_i\|^2},
\end{equation}
let $x_0$ be a square integrable $\HH$-valued random vector, 
and, for every $i\in\{1,\ldots,m\}$,
let $v_{i,0}$ be a square integrable $\GG_i$-valued random vector, 
let $\varepsilon \in \left]0, 1/(1+\beta)\right[$, let $(\gamma_n)_{n\in\NN}$ be a sequence in 
$[\varepsilon, (1-\varepsilon)/\beta]$. Set 
\begin{equation}
\label{e:sva2}
(\forall n\in\NN)\quad 
\begin{array}{l}
\left\lfloor
\begin{array}{l}
y_{1,n} = x_n - \gamma_n \big(Cx_n+ \sum_{i=1}^{m}L_{i}^*v_{i,n}+ a_{1,n}\big)\\
p_{1,n}=J_{\gamma_n A}(y_{1,n} + \gamma_nz) + b_{1,n}\\
\operatorname{for}\  i =1,\ldots, m\\
\left\lfloor
\begin{array}{l}
y_{2,i,n} = v_{i,n} + \gamma_n\big(L_ix_n - D_{i}^{-1}v_{i,n} + a_{2,i,n}\big)\\
p_{2,i,n} = J_{\gamma_nB_{i}^{-1}}(y_{2,i,n} -\gamma_nr_i) + b_{2,i,n}\\
q_{2,i,n} = p_{2,i,n} + \gamma_n\big(L_ip_{1,n} -D_{i}^{-1}p_{2,i,n} + c_{2,i,n}\big)\\
v_{i,n+1} = v_{i,n} - y_{2,i,n} + q_{2,i,n} 
\end{array}
\right.\\[2mm]
q_{1,n} = p_{1,n} -\gamma_n\big(Cp_{1,n} + \sum_{i=1}^m L^{*}_ip_{2,i,n} + c_{1,n}\big)\\
x_{n+1} = x_n - y_{1,n} + q_{1,n}.
\end{array}
\right.\\[2mm]
\end{array}
\end{equation}
Suppose that the following conditions hold 
for $\FF_n = \sigma( (x_k,( v_{i,k})_{1\leq i\leq m})_{0\leq k\leq n}$,
\begin{equation}\label{eq:coddd}
\begin{cases}
\sum_{n\in\NN}\sqrt{\E[||| (a_{1,n}, (a_{2,i,n})_{1\leq i\leq m})|||^2 | \FF_n]} < +\infty\\
\sum_{n\in\NN}\sqrt{\E[\| (b_{1,n}, (b_{2,i,n})_{1\leq i\leq m})|||^2 | \FF_n]} < +\infty\\
\sum_{n\in\NN}\sqrt{\E[||| (c_{1,n}, (c_{2,i,n})_{1\leq i\leq m})|||^2 | \FF_n]} < +\infty.
\end{cases}
\end{equation}
Then the following hold.
\begin{enumerate}
\item \label{t:Sv2i} 
$\sum_{n\in\NN}\E[\|x_n - p_{1,n}\|^2|\FF_n] < +\infty$ and 
$(\forall i\in \{1,\ldots,m\})\; \sum_{n\in\NN}\E[\|v_{i,n} - p_{2,i,n}\|^2|\FF_n] < +\infty$ $\boldsymbol{\mathsf{P}}$-a.s. 
 \item\label{t:Sv2ii}  
There exist a $\mathcal{P}$-valued random vector $\overline{x}$ and a $\mathcal{D}$-valued random vector  
$(\overline{v}_1,\ldots, \overline{v}_m)$  such that the following hold.
\begin{enumerate}
 \item \label{t:Sv2iia}
$x_n\weakly \overline{x}$ and 
$J_{\gamma_n A}(x_n - \gamma_n (Cx_n+ \sum_{i=1}^{m}L_{i}^*v_{i,n}) +\gamma_nz)\weakly \overline{x}$ $\boldsymbol{\mathsf{P}}$-a.s.
\item \label{t:Sv2iib}  $(\forall i\in \{1,\ldots,m\})\; v_{i,n} \weakly \overline{v}_i$ and 
$J_{\gamma_nB_{i}^{-1}}(v_{i,n} + \gamma_n\big(L_ix_n - D_{i}^{-1}v_{i,n})-\gamma_nr_i)  \weakly \overline{v}_i$ $\boldsymbol{\mathsf{P}}$-a.s.
\item \label{t:Sv2iic} Suppose that $A$ or $C$ is uniformly monotone at $\overline{x}(\omega)$
for every $\omega\in \widetilde{\boldsymbol{\Omega}}\subset \boldsymbol{\Omega}$ with 
$\boldsymbol{\mathsf{P}}(\widetilde{\boldsymbol{\Omega}}) =1$, then 
$x_n\to\overline{x}$ and 
$ J_{\gamma_n A}(x_n - \gamma_n (Cx_n+ \sum_{i=1}^{m}L_{i}^*v_{i,n}) +\gamma_nz)\to\overline{x}$ $\boldsymbol{\mathsf{P}}$-a.s.
\item \label{t:Sv2iid} 
Suppose that $B^{-1}_j$ or $D_{j}^{-1}$ is uniformly monotone at $\overline{v}_j(\omega)$
for every $\omega\in \widetilde{\boldsymbol{\Omega}}\subset \boldsymbol{\Omega}$ with 
$\boldsymbol{\mathsf{P}}(\widetilde{\boldsymbol{\Omega}}) =1$, 
for some $j\in\{1,\ldots,m\}$, then 
$v_{j,n}\to\overline{v}_j$ and $J_{\gamma_nB_{j}^{-1}}(v_{j,n} 
+ \gamma_n\big(L_jx_n - D_{j}^{-1}v_{j,n})-\gamma_nr_j)  \to\overline{v}_j$ $\boldsymbol{\mathsf{P}}$-a.s.
\end{enumerate}
\end{enumerate}
\end{corollary}
\begin{proof}
 Set
\begin{equation}
 \label{e:maximal1}
\begin{cases}
\AAA\colon\KKK\to 2^{\KKK}\colon
(x,v_1,\ldots,v_m)\mapsto (-z+ Ax)
\times(r_1 + B_{1}^{-1}v_1)\times\ldots\times(r_m + B^{-1}_{m}v_m)\\
\BB\colon\KKK\to \KKK\colon
(x,v_1,\ldots,v_m)\mapsto
\bigg(Cx+\sum_{i=1}^mL_{i}^*v_i,D_{1}^{-1}v_1-L_1x,\ldots,D_{m}^{-1}v_m-L_mx\bigg).
\end{cases}
\end{equation}
Since $\AAA$ is maximally monotone \cite[Propositions 20.22 and 20.23]{livre1},
$\BB$ is monotone and $\beta$-Lipschitzian \cite[Eq. (3.10)]{plc6} 
with $\dom\BB=\KKK$, $\AAA+\BB$ is maximally monotone~\cite[Corollary~24.24(i)]{livre1}. 
In addition, \cite[Propositions 23.15(ii) and 23.16]{livre1} yield
$(\forall \gamma \in \left]0,\pinf\right[)(\forall n\in\NN)(\forall (x,v_1,\ldots, v_m) \in\KKK)$
\begin{alignat}{2}
\label{e:coco} 
J_{\gamma  \AAA}(x,v_1,\ldots,v_m) = \Big(J_{\gamma  A }(x+\gamma z), 
\big(J_{\gamma B^{-1}_i}(v_i-\gamma r_i)\big)_{1\leq i\leq m}\Big).
\end{alignat}
It is shown in~\cite[Eq.~(3.12)]{plc6} and \cite[Eq.~(3.13)]{plc6} 
that under the condition~\eqref{e:fbfcond},
$\zer(\AAA + \BB )\neq\emp$. Moreover, \cite[\rm Eq.~(3.21)]{plc6} and 
~\cite[\rm Eq.~(3.22)]{plc6} yield
\begin{equation}
(\overline{x},\overline{v}_1,\ldots, \overline{v}_m) 
\in\zer(\AAA+ \BB)\Rightarrow \overline{x}\;\text{solves \eqref{primal12:41}}\;
\text{and}\;
(\overline{v}_1,\ldots, \overline{v}_m)\; \text{solves \eqref{dual12:41}}.
\end{equation}
Let us next set, for every $n\in\NN$, 
\begin{equation}
\label{e:e:cucucu}
 \begin{cases}
 \xx_n = (x_n,v_{1,n},\ldots,v_{m,n})\\
\yy_n = (y_{1,n},y_{2,1,n},\ldots, y_{2,m,n})\\
\pp_n = (p_{1,n},p_{2,1,n},\ldots, p_{2,m,n})\\
\qq_n = (q_{1,n},q_{2,1,n},\ldots, q_{2,m,n})
 \end{cases}
\quad \text{and}\quad 
\begin{cases}
 \aaa_n =  (a_{1,n},a_{2,1,n},\ldots, a_{2,m,n})\\
\bb_n = (b_{1,n},b_{2,1,n},\ldots, b_{2,m,n})\\
\cc_n = (c_{1,n},c_{2,1,n},\ldots, c_{2,m,n}).
\end{cases}
\end{equation}
Then our assumptions imply that 
\begin{equation}
\label{e:ds1a}
 \sum_{n\in\NN}\sqrt{\E[|||\aaa_n|||^2 |\FF_n ]}< \infty,\quad\sum_{n\in\NN}\sqrt{\E[|||\bb_n|||^2 |\FF_n ]}< \infty ,\quad 
\text{and}\quad \sum_{n\in\NN}\sqrt{\E[|||\cc_n|||^2 |\FF_n ]}< \infty.
\end{equation}
Furthermore, it follows from the definition of $\BB$, \eqref{e:coco}, and  \eqref{e:e:cucucu} 
that \eqref{e:sva2} can be rewritten in $\KKK$ as 
 \begin{equation}
\label{e:Tsengb}
(\forall n\in\NN)\quad
\begin{array}{l}
\left\lfloor
\begin{array}{l}
\yy_n= \xx_n-\gamma_n(\BB\xx_n+\aaa_n)\\[1mm]
\pp_n = J_{\gamma_n \AAA}\yy_n + \bb_n\\
\qq_n = \pp_n -\gamma_n(\BB\pp_n+\cc_n)\\
\xx_{n+1} = \xx_n -\yy_n +\qq_n,
\end{array}
\right.\\[2mm]
\end{array}
\end{equation}
which is \eqref{e:Tsenga}. Moreover,  
every specific conditions in Corollary \ref{l:Tsenglemma} are satisfied.

\ref{t:Sv2i}: By Corollary \ref{l:Tsenglemma}\ref{Tsengi}, 
$\sum_{n\in\NN}\E[|||\xx_n-\pp_{n} |||^2|\FF_n ]< \infty$.

\ref{t:Sv2ii}(a)\&\ref{t:Sv2iib}: It follows from Corollary \ref{l:Tsenglemma}\ref{Tsengii} that 
\begin{equation}
x_n\weakly \overline{x}\quad 
\text{and}\quad (\forall i\in\{1,\ldots,m\})\quad v_{i,n}\weakly\overline{v}_i\quad \boldsymbol{\mathsf{P}}-a.s.
\end{equation}
Corollary \ref{l:Tsenglemma}\ref{Tsengii} shows that 
$(\overline{x},\overline{v}_1, \ldots, \overline{v}_m)\in\zer(\AAA+\BB)$. Hence, it follows
from~\cite[Eq~(3.19)]{plc6} that
 $(\overline{x},\overline{v}_1, \ldots, \overline{v}_m)$ satisfies the inclusions
\begin{equation}
\label{e:dualsol}
\begin{cases}
-\sum_{i = 1}^m L_{i}^*\overline{v}_{i} - C\overline{x}\in -z+A\overline{x} \\
(\forall i\in\{1,\ldots,m\})\;L_i\overline{x}- D^{-1}_i\overline{v}_{i}\in r_i+ B^{-1}_i\overline{v}_{i}.
\end{cases}
\end{equation}
For every $n\in\NN$ and every $i\in\{1,\ldots,m\}$, set
\begin{equation}
\label{e:sva21}
 \begin{cases}
  \widetilde{y}_{1,n} = x_n - \gamma_n \big(Cx_n+ \sum_{i=1}^{m}L_{i}^*v_{i,n}\big)\\
\widetilde{p}_{1,n}=J_{\gamma_n A}(\widetilde{y}_{1,n} + \gamma_nz)\\
 \end{cases}
\quad \text{and}\quad
\begin{cases}
\widetilde{y}_{2,i,n} = v_{i,n} + \gamma_n\big(L_ix_n - D_{i}^{-1}v_{i,n}\big)\\
\widetilde{p}_{2,i,n} = J_{\gamma_n B_{i}^{-1}}(\widetilde{y}_{2,i,n} -\gamma_nr_i).\\
\end{cases}
\end{equation}
We note that \eqref{e:ds1a} implies that 
\begin{equation} 
\label{e:alm}
\E[|||\aaa_n|||^2 |\FF_n ]\to 0, \quad  \E[|||\bb_n|||^2 |\FF_n ]\to 0 \quad \text{and}\quad \E[|||\cc_n|||^2 |\FF_n ]\to 0
\quad\text{$\boldsymbol{\mathsf{P}}$-a.s.} 
\end{equation}
Then, using \cite[Corollary 23.10]{livre1}, we get
\begin{alignat}{2}
\begin{cases}
\| \widetilde{p}_{1,n} - p_{1,n}\| \leq \|b_{1,n}\| + \beta^{-1}\|a_{1,n}\|, \\
 (\forall i\in \{1,\ldots,m\})
\quad\| \widetilde{p}_{2,i,n}-p_{2,i,n}\|\leq \|b_{2,i,n}\| + \beta^{-1}\|a_{2,i,n}\|, 
\end{cases}
\end{alignat} 
which and \eqref{e:alm} imply that 
\begin{alignat}{2}
\begin{cases}
\E[\| \widetilde{p}_{1,n} - p_{1,n}\|^2|\FF_n] \leq 2 \E[\|b_{1,n}\|^2 + \beta^{-2}\|a_{1,n}\|^2|\FF_n] \to 0
\;\text{$\boldsymbol{\mathsf{P}}$-a.s.} \\
 (\forall i\in \{1,\ldots,m\})
\;\E[ \| \widetilde{p}_{2,i,n}-p_{2,i,n}\|^2|\FF_n] \leq2\E[\|b_{2,i,n}\|^2 + \beta^{-2}\|a_{2,i,n}\|^2|\FF_n]\to 0\;
\text{$\boldsymbol{\mathsf{P}}$-a.s.}
\end{cases}
\end{alignat} 
Since $(x,v_1\ldots,v_m)\mapsto J_{\gamma_n A}(x - \gamma_n \big(Cx+ \sum_{i=1}^{m}L_{i}^*v_{i}) +\gamma_nz)$
is continuous from $\KKK\to \HH$, $\widetilde{p}_{1,n}$ is $\FF_n$-measurable. By the same way,
for every $i\in\{1,\ldots,m\}$,
$\widetilde{p}_{2,i,n}$ is $\FF_n$-measurable. 
In turn, by \ref{t:Sv2i},\ref{t:Sv2iia}, and \ref{t:Sv2iib}, we obtain
\begin{equation}
\label{e:hoaquangoc}
\begin{cases}
\|\widetilde{p}_{1,n} - x_n\|^2 = \E[\|\widetilde{p}_{1,n} - x_n\|^2|\FF_n]
\leq 2\E[\|p_{1,n} - x_n\|^2 +\|\widetilde{p}_{1,n} - p_{1,n}\|^2|\FF_n]\to 0\; \text{$\boldsymbol{\mathsf{P}}$-a.s.}\\
(\forall i\in\{1,\ldots,m\})\;
\|\widetilde{p}_{2,i,n}-v_{i,n}\|^2 = \E[\|\widetilde{p}_{2,i,n}-v_{i,n}\|^2 |\FF_n]\\ 
\hspace{5.5cm}\leq 2 \E[\|\widetilde{p}_{2,i,n}-p_{2,i,n}\|^2+\|p_{2,i,n}-v_{i,n}\|^2 |\FF_n] \to 0\;\text{$\boldsymbol{\mathsf{P}}$-a.s.}\\
  \widetilde{p}_{1,n} \weakly \overline{x}\quad \text{$\boldsymbol{\mathsf{P}}$-a.s.}\quad
\text{and}\quad
(\forall i\in\{1,\ldots,m\})\quad \widetilde{p}_{2,i,n}\weakly \overline{v}_i\quad \text{$\boldsymbol{\mathsf{P}}$-a.s.}
\end{cases}
\end{equation}
 \ref{t:Sv2iic}:
We derive from
\eqref{e:sva21} that 
\begin{equation}
\label{e:dualsola}
(\forall n\in\NN)\quad
 \begin{cases}
 \gamma^{-1}_n(x_n-\widetilde{p}_{1,n})
-\sum_{i = 1}^m L_{i}^*v_{i,n} - Cx_n \in -z + A\widetilde{p}_{1,n}\\
(\forall i\in \{1,\ldots,m\})\;
\gamma^{-1}_n(v_{i,n}-\widetilde{p}_{2,i,n}) + L_ix_n  
- D^{-1}_i v_{i,n}  \in r_i + B_{i}^{-1}\widetilde{p}_{2,i,n}.\\
 \end{cases}
\end{equation}
Let $\boldsymbol{\Omega_3}$ be the set of all $\omega\in \boldsymbol{\Omega}$
such that  $(x_n(\omega)-\overline{x}(\omega))_{n\in\NN},
(\widetilde{p}_{1,n}(\omega) -\overline{x}(\omega))_{n\in\NN}$ and $ (\forall i\in\{1,\ldots,m\})\; 
(v_{i,n}(\omega)-\overline{v}_i(\omega))_{n\in\NN}, 
(\widetilde{p}_{2,i,n}(\omega) - \overline{v}_i(\omega))_{n\in\NN}$ are bounded, and 
$(\forall i\in\{1,\ldots,m\})\; \widetilde{p}_{2,i,n}(\omega) - \overline{v}_{i,n}(\omega)\to 0$,  
$\widetilde{p}_{1,n}(\omega) - x_n(\omega)\to 0$.
Set $ \boldsymbol{\Omega_4} = \boldsymbol{\Omega_3}\cap \boldsymbol{\widetilde{\Omega}}$. 
Then  $\boldsymbol{\Omega_4}$ has probability $1$. Now fix $\omega \in \boldsymbol{\Omega_4}$.
Since $A$ is uniformly monotone at $\overline{x}(\omega)$, 
using \eqref{e:dualsol} and \eqref{e:dualsola},
there exists an increasing function
$\phi_A\colon\left[0,+\infty\right[\to\left[0,+\infty\right]$ 
vanishing only at $0$ such that, for every $n\in\NN$,  
\begin{alignat}{2}
 \phi_A(\|\widetilde{p}_{1,n}(\omega) -\overline{x}(\omega)\|) &\leqslant 
\scal{\widetilde{p}_{1,n}(\omega)-\overline{x}(\omega)}{\gamma^{-1}(x_{n}(\omega)-\widetilde{p}_{1,n}(\omega)) 
- \sum_{i = 1}^m(L_{i}^*v_{i,n}(\omega) - L_{i}^*\overline{v}_{i}(\omega)) }\notag\\
&\quad-\chi_n(\omega)\notag\\
&= \scal{\widetilde{p}_{1,n}(\omega)-\overline{x}(\omega)}{\gamma^{-1}_n(x_{n}(\omega)-\widetilde{p}_{1,n}(\omega))}
 -\chi_n(\omega)\notag\\
&\quad- \sum_{i = 1}^m\scal{\widetilde{p}_{1,n}(\omega)
-\overline{x}(\omega)}{L_{i}^*v_{i,n}(\omega) - L_{i}^*\overline{v}_{i}(\omega)}
\label{cz1},
\end{alignat}
where we denote $\big(\forall n\in\NN\big)\; 
\chi_n(\omega) = \scal{\widetilde{p}_{1,n}(\omega) -\bar{x}(\omega)}{Cx_n(\omega) -C\bar{x}(\omega)}$.
Since $(B_{i}^{-1})_{1\leq i\leq m}$ are monotone, for every
$i\in\{1,\ldots,m\}$, we obtain
\begin{alignat}{2}
(\forall n\in\NN)\quad
0 &\leqslant \scal{\widetilde{p}_{2,i,n}(\omega) - \overline{v}_{i}(\omega)}{L_ix_{n}(\omega) 
+\gamma^{-1}_n(v_{i,n}(\omega)-\widetilde{p}_{2,i,n}(\omega)) -L_i\overline{x}(\omega)} -\beta_{i,n}(\omega)\notag\\
&= \scal{\widetilde{p}_{2,i,n}(\omega) - \overline{v}_i(\omega)}{L_i(x_{n}(\omega) -\overline{x}(\omega)) 
+\gamma^{-1}_n(v_{i,n}(\omega)-\widetilde{p}_{2,i,n}(\omega))} 
- \beta_{i,n}(\omega),
\label{cz2}
\end{alignat}
where $\big(\forall n\in\NN\big)\;
 \beta_{i,n}(\omega) = \scal{\widetilde{p}_{2,i,n}(\omega) - \overline{v}_{i}(\omega)}{D_{i}^{-1}v_{i,n}(\omega) - D_{i}^{-1}
\bar{v}_{i}(\omega)}$.
Now, adding~\eqref{cz2} from $i = 1$ to $i = m$ and~\eqref{cz1}, we obtain,
for every $n\in\NN$,
\begin{alignat}{2}\label{cz3}
 \phi_A(\|\widetilde{p}_{1,n}(\omega) -\overline{x}(\omega)\|) 
&\leq\scal{\widetilde{p}_{1,n}(\omega)-\overline{x}(\omega)}{\gamma^{-1}_n(x_{n}(\omega)-\widetilde{p}_{1,n}(\omega))}\notag\\
&\quad+\scal{\widetilde{p}_{1,n}(\omega)-\overline{x}(\omega)}{ \sum_{i=1}^m L_{i}^*(\widetilde{p}_{2,i,n}(\omega)-v_{i,n}(\omega))}
\notag \\
&\quad+\sum_{i = 1}^m \scal{\widetilde{p}_{2,i,n}(\omega) - \overline{v}_i(\omega)}{L_i(x_{n}(\omega) -\widetilde{p}_{1,n}(\omega)) 
+\gamma^{-1}_n(v_{i,n}(\omega)-\widetilde{p}_{2,i,n}(\omega))}\notag\\
&\quad-\chi_n(\omega) - \sum_{i = 1}^{m}\beta_{i,n}(\omega).
\end{alignat}
For every $n\in\NN$ and every $ i\in \{1,\ldots,m\}$,
we expand $\chi_n(\omega)$ and $\beta_{i,n}(\omega)$ as
\begin{equation}
\label{e:expa}
\begin{cases}
\chi_n(\omega) = \scal{x_n(\omega)-\overline{x}(\omega)}{ Cx_n(\omega)- C\overline{x}(\omega)}
 + \scal{\widetilde{p}_{1,n}(\omega) -x_n(\omega)}{ Cx_n(\omega)- C\overline{x}(\omega)},\\
\;\beta_{i,n}(\omega) =
\scal{v_{i,n}(\omega)-\overline{v}_i(\omega)}{ D_{i}^{-1}v_{i,n}(\omega)- D_{i}^{-1}\overline{v}_i(\omega)}\\
 \hspace{6cm}+ \scal{\widetilde{p}_{2,i,n}(\omega) - v_{i,n}(\omega)}{ D_{i}^{-1}v_{i,n}(\omega)- D_{i}^{-1}\overline{v}_i(\omega)}.
\end{cases}
\end{equation}
By  monotonicity of $C$ and $(D_{i}^{-1})_{1\leq i\leq m}$,
\begin{equation}
(\forall n\in\NN)\quad
 \begin{cases}
\scal{x_n(\omega)-\overline{x}(\omega)}{ Cx_n(\omega)- C\overline{x}(\omega)}\geq 0,\\
(\forall i\in \{1,\ldots,m\})\; 
\scal{v_{i,n}(\omega)-\overline{v}_i(\omega)}{ D_{i}^{-1}v_{i,n}(\omega)- D_{i}^{-1}\overline{v}_i(\omega)}\geq 0.
 \end{cases}
\end{equation}
Therefore, for every $n\in\NN$, we derive from \eqref{e:expa} and \eqref{cz3} that
\begin{alignat}{2}
\label{e:concao}
 \phi_A(\|\widetilde{p}_{1,n}(\omega) -\overline{x}(\omega)\|)&\leq
 \phi_A(\|\widetilde{p}_{1,n}(\omega) -\overline{x}(\omega)\|) 
+ \scal{x_n(\omega)-\overline{x}(\omega)}{ Cx_n(\omega)- C\overline{x}(\omega)}\notag\\
&\quad
+ \sum_{i=1}^m \scal{v_{i,n}(\omega)-\overline{v}_i(\omega)}{ D_{i}^{-1}v_{i,n}(\omega)- D_{i}^{-1}\overline{v}_i(\omega)}\notag\\
&\leq\scal{\widetilde{p}_{1,n}(\omega)-\overline{x}(\omega)}{\gamma^{-1}_n(x_{n}(\omega)-\widetilde{p}_{1,n}(\omega))}\notag\\
&\quad+\scal{\widetilde{p}_{1,n}(\omega)-\overline{x}(\omega)}{ \sum_{i=1}^m L_{i}^*(\widetilde{p}_{2,i,n}(\omega)-v_{i,n}(\omega))}
\notag \\
&\quad+\sum_{i = 1}^m \scal{\widetilde{p}_{2,i,n}(\omega) - \overline{v}_i(\omega)}{ 
L_i(x_{n}(\omega) -\widetilde{p}_{1,n}(\omega)) 
+\gamma^{-1}_n(v_{i,n}(\omega)-\widetilde{p}_{2,i,n}(\omega))}\notag\\
&\quad - \scal{\widetilde{p}_{1,n}(\omega) -x_n(\omega)}{ Cx_n(\omega)- C\overline{x}(\omega)} \notag\\
&\quad-\sum_{i=1}^m\scal{\widetilde{p}_{2,i,n}(\omega) - v_{i,n}(\omega)}{ D_{i}^{-1}v_{i,n}(\omega)
- D_{i}^{-1}\overline{v}_i(\omega)}.
\end{alignat}
We set
\begin{equation}
 \zeta(\omega) = \max_{1\leq i\leq m} \sup_{n\in\NN}\{\|x_n(\omega) -\overline{x}(\omega)||,
\|\widetilde{p}_{1,n}(\omega) -\overline{x}(\omega)\|, 
\|v_{i,n}(\omega)-\overline{v}_i(\omega)\|, \|\widetilde{p}_{2,i,n}(\omega) - \overline{v}_i(\omega)\| \}. 
\end{equation}
Then it follows from the definition of $\boldsymbol{\Omega_4}$ that
 $\zeta(\omega) < \infty$, and  from our assumption that
$(\forall n\in\NN)\;\gamma^{-1}_n\leq \varepsilon^{-1}$. 
Therefore, using the Cauchy-Schwarz inequality, and the Lipschitzianity 
of $C$ and $(D^{-1}_i)_{1\leq i\leq m}$,
we derive from \eqref{e:concao} that
\begin{alignat}{2}\label{cz3e}
 \phi_A(\|\widetilde{p}_{1,n}(\omega) -\overline{x}(\omega)\|) 
&\leq \varepsilon^{-1}\zeta\|x_{n}(\omega)-\widetilde{p}_{1,n}(\omega)\|
+\zeta\sum_{i = 1}^m\big(\|L_i\|\; \|x_{n}(\omega) -\widetilde{p}_{1,n}(\omega)\| 
\notag \\
&\quad+ \varepsilon^{-1}\| v_{i,n}-\widetilde{p}_{2,i,n}\|\big)
+\zeta(\omega)\bigg(\sum_{i=1}^m \|L_{i}^*\| \|\widetilde{p}_{2,i,n}(\omega)-v_{i,n}(\omega)\|\notag\\
&\quad+\nu_0\|\widetilde{p}_{1,n}(\omega) -x_n(\omega)\| 
+ \sum_{i=1}^m \nu_i \|\widetilde{p}_{2,i,n}(\omega) - v_{i,n}(\omega)\|\bigg) \notag\\
&\to 0.
\end{alignat}
We deduce from~\eqref{cz3e} and \eqref{e:hoaquangoc} 
that $\phi_A(\|\widetilde{p}_{1,n}(\omega) -\overline{x}(\omega)\|)\to 0$, which implies that 
$\widetilde{p}_{1,n}(\omega)\to \overline{x}(\omega)$. In turn,
$x_n(\omega) \to \overline{x}(\omega)$.
Likewise, if $C$
is uniformly monotone at $\overline{x}(\omega)$, 
there exists an increasing function $\phi_C\colon\left[0,+\infty\right[\to\left[0,+\infty\right]$ that vanishes
only at $0$ such that
\begin{alignat}{2}\label{cz3ef}
 \phi_C(\|x_{n}(\omega) -\overline{x}(\omega)\|) 
&\leq \varepsilon^{-1}\zeta\|x_{n}(\omega)-\widetilde{p}_{1,n}(\omega)\|
+\zeta\sum_{i = 1}^m\big(\|L_i\|\; \|x_{n}(\omega) -\widetilde{p}_{1,n}(\omega)\| 
\quad\notag \\
&\quad+ \varepsilon^{-1}\| v_{i,n}(\omega)-\widetilde{p}_{2,i,n}(\omega)\|\big)
+\zeta\bigg(\sum_{i=1}^m \|L_{i}^*\| \|\widetilde{p}_{2,i,n}(\omega)-v_{i,n}(\omega)\|\notag\\
&\quad+\nu_0\|\widetilde{p}_{1,n}(\omega) -x_n(\omega)\| 
+ \sum_{i=1}^m \nu_i \|\widetilde{p}_{2,i,n}(\omega) - v_{i,n}(\omega)\|\bigg) \notag\\
&\to 0,
\end{alignat}
in turn,  $x_n(\omega) \to \overline{x}(\omega)$.

\ref{t:Sv2iid}: Proceeding
as in the proof of \ref{t:Sv2iic}, we obtain the conclusions.
\end{proof}

We provide an application  to minimization problems in \cite[Section 4]{plc6} which cover
a wide class of convex optimization problems in the literature.
We recall that 
the infimal convolution of the two functions $f$ and $g$ from $\HH$ to $\left]-\infty,+\infty\right]$ is 
\begin{equation}
 f\;\vuo\; g\colon x \mapsto \inf_{y\in\HH}(f(y)+g(x-y)).
\end{equation}
 The proximity operator of $f\in\Gamma_0(\HH)$,
 denoted by $\prox_{f}$, which maps each point $x\in\HH$ to the unique minimizer 
of the function $f +\frac12\|x-\cdot\|^2$. 
\begin{example}
\label{ex:prob1}
Let $m$ be a strictly positive integer.
Let $\HH$ be a real separable Hilbert space,
let $z\in\HH$, let $f\in\Gamma_0(\HH)$,
let $h\colon \HH\to\RR$ be 
convex differentiable function with $\nu_0$-Lipschitz continuous gradient,
for some $\nu_0 \in \left]0,+\infty \right[$.
For every $k\in\{1,\ldots, m\}$, 
let $(\GG_k,\scal{\cdot}{\cdot})$ be a real separable Hilbert space, 
let $r_k \in \GG_k$, 
let $g_{k}\in\Gamma_0(\GG_k)$, let $\ell_k\in\Gamma_0(\GG_k)$
be $1/\nu_k$-strongly convex, for some $\nu_k \in \left]0,+\infty \right[$.
For  every $k\in\{1,\ldots,m\}$,
let $L_{k}\colon\HH \to\GG_k$ 
be a  bounded linear operator. 
The primal problems is to 
\begin{alignat}{2} 
\label{primal2}
&\underset{x\in\HH}{\text{minimize}}\big(f(x) - \scal{x}{z}\big)+
 \sum_{k=1}^m\big(\ell_k\;\vuo\; g_{k})
\big)\bigg( L_{k}x -r_k\bigg) 
+h(x),
\end{alignat}
and the dual problem is to 
\begin{equation}
\label{dual2} 
 \underset{v_1\in\GG_1,\ldots,v_m\in\GG_m}{\text{minimize}} \;
(f^*\;\vuo\; h^*)\bigg(z-\sum_{i=1}^mL_{i}^*v_i\bigg) +
\sum_{i=1}^m \big(g^{*}_i(v_i)+\ell^{*}_i(v_i) +\scal{v_i}{r_i}\big).
\end{equation}
We denote by $\mathcal{P}_1$ and $\mathcal{D}_1$ be the set of solutions 
to \eqref{primal2} and \eqref{dual2}, respectively.
\end{example}

\begin{corollary}
\label{t:Sv2a}
 In Example \ref{ex:prob1},  suppose that
\begin{equation}\label{IVe:ranex6}
z\in\ran\bigg( \partial f + 
\sum_{i=1}^mL^{*}_i
\big((\partial g_i\;\vuo\;\partial\ell_i)(L_i\cdot-r_i)\big)
+\nabla h\bigg).
\end{equation}
Let $(a_{1,n})_{n\in\NN}, (b_{1,n})_{n\in\NN}$, and $(c_{1,n})_{n\in\NN}$
be  sequences of square integrable $\HH$-valued random vectors, and for every $i\in\{1,\ldots,m\}$, 
let $(a_{2,i,n})_{n\in\NN}, (b_{2,i,n})_{n\in\NN}$, and $(c_{2,i,n})_{n\in\NN}$ 
be  sequences of square integrable $\GG_i$-valued random vectors. Furthermore, set 
\begin{equation}
 \beta = \max\{\nu_0,\nu_1,\ldots, \nu_m\} + \sqrt{\sum_{i=1}^m\|L_i\|^2},
\end{equation}
let $x_0$ be a square integrable $\HH$-valued random vector, 
and, for every $i\in\{1,\ldots,m\}$,
let $v_{i,0}$ be a square integrable $\GG_i$-valued random vector, 
let $\varepsilon \in \left]0, 1/(1+\beta)\right[$, let $(\gamma_n)_{n\in\NN}$ be a sequence in 
$[\varepsilon, (1-\varepsilon)/\beta]$. Set 
\begin{equation}
\label{e:sva2a}
(\forall n\in\NN)\quad 
\begin{array}{l}
\left\lfloor
\begin{array}{l}
y_{1,n} = x_n - \gamma_n \big(\nabla h(x_n)+ \sum_{i=1}^{m}L_{i}^*v_{i,n}+ a_{1,n}\big)\\
p_{1,n}=\prox_{\gamma_n f}(y_{1,n} + \gamma_nz) + b_{1,n}\\
\operatorname{for}\  i =1,\ldots, m\\
\left\lfloor
\begin{array}{l}
y_{2,i,n} = v_{i,n} + \gamma_n\big(L_ix_n - \nabla \ell_{i}^{*}(v_{i,n}) + a_{2,i,n}\big)\\
p_{2,i,n} = \prox_{\gamma_ng_{i}^{*}}(y_{2,i,n} -\gamma_nr_i) + b_{2,i,n}\\
q_{2,i,n} = p_{2,i,n} + \gamma_n\big(L_ip_{1,n} -\nabla \ell_{i}^{*}(p_{2,i,n}) + c_{2,i,n}\big)\\
v_{i,n+1} = v_{i,n} - y_{2,i,n} + q_{2,i,n} 
\end{array}
\right.\\[2mm]
q_{1,n} = p_{1,n} -\gamma_n\big(\nabla h(p_{1,n}) + \sum_{i=1}^m L^{*}_ip_{2,i,n} + c_{1,n}\big)\\
x_{n+1} = x_n - y_{1,n} + q_{1,n}.
\end{array}
\right.\\[2mm]
\end{array}
\end{equation}
Suppose that the following conditions hold 
for $\FF_n = \sigma( (x_k,( v_{i,k})_{1\leq i\leq m})_{0\leq k\leq n}$,
\begin{equation}\label{eq:coddda}
\begin{cases}
\sum_{n\in\NN}\sqrt{\E[||| (a_{1,n}, (a_{2,i,n})_{1\leq i\leq m})|||^2 | \FF_n]} < +\infty\\
\sum_{n\in\NN}\sqrt{\E[\| (b_{1,n}, (b_{2,i,n})_{1\leq i\leq m})|||^2 | \FF_n]} < +\infty\\
\sum_{n\in\NN}\sqrt{\E[||| (c_{1,n}, (c_{2,i,n})_{1\leq i\leq m})|||^2 | \FF_n]} < +\infty.
\end{cases}
\end{equation}
Then the following hold.
\begin{enumerate}
\item \label{t:Sv2ia} $\sum_{n\in\NN}\E[\|x_n - p_{1,n}\|^2|\FF_n] < +\infty$ and 
$(\forall i\in \{1,\ldots,m\})\; \sum_{n\in\NN}\E[\|v_{i,n} - p_{2,i,n}\|^2|\FF_n] < +\infty$. 
 \item\label{t:Sv2iia}  
There exist a $\mathcal{P}_1$-valued random vector $\overline{x}$ and a $\mathcal{D}_1$-valued random vector  
$(\overline{v}_1,\ldots, \overline{v}_m)$  such that the following hold.
\begin{enumerate}
 \item \label{t:Sv2iia}
$x_n\weakly \overline{x}$ and 
$\prox_{\gamma_n f}(x_n - \gamma_n (\nabla h(x_n)+ \sum_{i=1}^{m}L_{i}^*v_{i,n}) +\gamma_nz)\weakly \overline{x}$ $\boldsymbol{\mathsf{P}}$-a.s.
\item \label{t:Sv2iib}  $(\forall i\in \{1,\ldots,m\})\; v_{i,n} \weakly \overline{v}_i$ and 
$\prox_{\gamma g^{*}_i}(v_{i,n} + \gamma_n\big(L_ix_n - \nabla \ell^{*}_{i}(v_{i,n})-\gamma_nr_i)  \weakly \overline{v}_i$ $\boldsymbol{\mathsf{P}}$-a.s.
\item \label{t:Sv2iic} Suppose that $f$ or $\nabla h$ is uniformly convex at $\overline{x}(\omega)$
for every $\omega\in \widetilde{\boldsymbol{\Omega}}\subset \boldsymbol{\Omega}$ with 
$\boldsymbol{\mathsf{P}}(\widetilde{\boldsymbol{\Omega}}) =1$, then 
$x_n\to\overline{x}$ and 
$ \prox_{\gamma_n f}(x_n - \gamma_n (\nabla h(x_n)+ \sum_{i=1}^{m}L_{i}^*v_{i,n}) +\gamma_nz)\to\overline{x}$ $\boldsymbol{\mathsf{P}}$-a.s.
\item \label{t:Sv2iid} 
Suppose that $g^{*}_j$ or $\ell_{j}^{*}$ is uniformly convex at $\overline{v}_j(\omega)$
for every $\omega\in \widetilde{\boldsymbol{\Omega}}\subset \boldsymbol{\Omega}$ with 
$\boldsymbol{\mathsf{P}}(\widetilde{\boldsymbol{\Omega}}) =1$, 
for some $j\in\{1,\ldots,m\}$, then 
$v_{j,n}\to\overline{v}_j$ and $\prox_{\gamma_ng^{*}_j}(v_{j,n} 
+ \gamma_n\big(L_jx_n - \nabla\ell^{*}_{j}(v_{j,n})-\gamma_nr_j)  \to\overline{v}_j$ $\boldsymbol{\mathsf{P}}$-a.s.
\end{enumerate}
\end{enumerate}
\end{corollary}
\begin{proof}
Using the same argument as in the proof \cite[Theorem 4.2]{plc6}, the conclusions follows from Corollary \ref{t:Sv2}.
\end{proof}
\begin{remark} Here are some comments.
\begin{enumerate}
\item By using Remark \ref{r:1}, an extension of Corollary \ref{t:Sv2} to the variable metric setting is straightforward.
\item
Almost sure convergence for some primal-dual splitting methods solving composite monotone inclusions 
and composite minimization problems are also presented in \cite{plc14,JCP14}.
\item In the deterministic setting and 
in the case when each $\ell_k$ is the indicator function of $\{0\}$,  and $(\forall k \in\{1,\ldots, m\}) r_k =0$, and $z=0$,
a preconditioned algorithm for solving \eqref{primal2}  can be found in \cite{Pesquet12}.
\end{enumerate}
\end{remark}
\noindent{{\bfseries Acknowledgement.}}
I thank Professor Patrick L. Combettes  for helpful discussions. I thank the referees 
 for their suggestions and correction which helped to improve the first version of the manuscript.
This work is funded by Vietnam National Foundation for Science and Technology
Development (NAFOSTED) under Grant No. 102.01-2014.02.

\end{document}